\documentclass{article}
\usepackage[utf8]{inputenc}

\usepackage[a4paper, left=2cm, right=2cm, top=2cm]{geometry}
\usepackage{xcolor}
\usepackage{amsmath,amsfonts,amsthm,amssymb}
\newtheorem{assumption}{Assumption}
\newtheorem{theorem}{Theorem}
\newtheorem{proposition}[theorem]{Proposition}%
\newtheorem{lemma}[theorem]{Lemma}%

\usepackage{algorithm}%
\usepackage{algorithmicx}
\usepackage{graphicx}

\DeclareMathOperator*{\argmax}{arg\,max}
\DeclareMathOperator*{\argmin}{arg\,min}

\DeclareMathOperator{\closed}{clos}

\DeclareMathOperator{\conv}{conv}

\newcommand{\norm}[1]{\left\lVert#1\right\rVert}

\providecommand{\keywords}[1]
{
  \small	
  \textbf{\textit{Keywords---}} #1
}

\title{Hermite kernel surrogates for the value function of high-dimensional nonlinear optimal control problems}
\author{Tobias Ehring$^{1^{\Huge *}}$ and  Bernard Haasdonk$^{1}$  \\
        \small $^{1}$Institute of Applied Analysis and Numerical Simulation, \\  \small University of Stuttgart, Pfaffenwaldring 57, Stuttgart, 70569, Baden-Württemberg, Germany. \\
        \small $^{*}$Corresponding author. E-mail: \\ \small ehringts@mathematik.uni-stuttgart.de
}
\date{}

\begin{document}

\maketitle
  \begin{abstract}
\noindent Numerical methods for the  optimal feedback control of high-dimensional dynamical systems typically suffer from the curse of dimensionality. In the current presentation, we devise a mesh-free data-based approximation method for the value function of optimal control problems, which partially mitigates the dimensionality problem. The method is based on a greedy Hermite kernel interpolation scheme and incorporates context-knowledge by its structure. Especially, the value function surrogate is elegantly enforced to be 0 in the target state, non-negative and constructed as a correction of a linearized model. The algorithm is proposed in a matrix-free way, which circumvents the large-matrix-problem for multivariate Hermite interpolation. For finite time horizons, both convergence of the   surrogate to the value function as well as for the surrogate vs. the optimal controlled dynamical system are proven. Experiments support the effectiveness of the scheme, using among others a new academic model that has a scalable dimension and an explicitly given value function. It may also be useful for the community to validate other optimal control approaches.
  \end{abstract}

\section{Introduction}\label{sec1}
The feedback control of dynamical systems is an important type of automation that is indispensable in modern industrial plants, for example. Furthermore, there are applications in which the control must be optimal to minimize energy consumption or costs. An optimal feedback rule can be based on a mathematical model, an optimal control problem (OCP), which shall be in our consideration of the following form
\begin{gather}
\min_{\mathbf{u} \in \mathcal{U}_{\infty}} J(\mathbf{u}) = \min_{\mathbf{u} \in \mathcal{U}_{\infty}} \int_0^{\infty} r(\mathbf{x}(s)) + \mathbf{u}(s)^{\top}R \mathbf{u}(s) \,\text{d}s  \label{eq:MP} \\ 
\text{ subject to } \dot{\mathbf{x}}(s) = f(\mathbf{x}(s))+g(\mathbf{x}(s))\mathbf{u}(s)  \text{ and }\mathbf{x}(0) = x_0 \in \mathbb{R}^N \label{eq:ODE2}
\end{gather}
with $\mathcal{U}_{\infty} := \{ \mathbf{u}: \left[0, \infty \right) \rightarrow  \mathbb{R}^M \, \vert \, \mathbf{u}\text{ measurable} \}$ being the space of all possible control signals. Note that we indicate time-dependent scalar or vectorial functions with bold symbols to distinguish them from time-independent quantities.  The OCP has an infinite time horizon, which is typical, for example, in mathematical economics when considering economic sustainability or economic growth \cite{Sethi2021}. Moreover, the dynamics of the trajectory $\mathbf{x} : [0,\infty) \rightarrow \mathbb{R}^N$ is affine with respect to the control signal $\mathbf{u}$. This is common when the ordinary differential equation (ODE) results from a semi-discretized partial differential equation (PDE) where the controller could be an external force. The function $J:\mathcal{U}_{\infty} \rightarrow  \mathbb{R}_+$ is  the cost function, which determines the goal of optimal control. It consists of the integral over the running payoff $r \in C(\mathbb{R}^N,\mathbb{R}_+)$ and a quadratic control term, involving a positive definite matrix  $R \in \mathbb{R}^{M \times M}$. It is assumed that the preimage of $r$ for $0$ is only $0$, i.e. $r^{-1}(\{0\}) = \{0\}$ and $f(0) = 0$, leading to the desired target position in the zero state. Furthermore, as the constraints in \eqref{eq:ODE2} determine the solution of the dynamical system for a given $\mathbf{u}$, the cost function depends only on $\mathbf{u}$. The optimal process of the infinite horizon OCP is denoted by $(\mathbf{x}^*,\mathbf{u}^*)$. For a compact set of initial states $ \mathcal{A} \subset \mathbb{R}^N $ it is assumed throughout the paper that for every $ x_0 \in \mathcal{A} $ there is an optimal process. Existence results for infinite horizon OCP can be found in \cite{Dmitruk2005}. \\
\newline
Two possibilities for solving \eqref{eq:MP}--\eqref{eq:ODE2} are given by  open-loop and  closed-loop control.  In an open-loop process, the optimal signal is determined in advance as a function $\mathbf{u} : [0,\infty) \longrightarrow \mathbb{R}^M$   and no further adaptation of the signal to the current state is possible in an ongoing process, making it unsuitable for many real-world applications where unpredictable disturbances may occur. This requires a closed-loop that maps the current state to the current optimal control signal, which is in fact a much more general tool since it can be used to solve \eqref{eq:MP}--\eqref{eq:ODE2} for many initial states. The optimal feedback rule is denoted by $\mathcal{K}(\,\cdot \,;\nabla v) : \mathbb{R}^N \longrightarrow \mathbb{R}^M$. It reduces the OCP to an ODE by setting $\mathbf{u}(s) = \mathcal{K}(\mathbf{x}(s);\nabla v)$ and inserting it into \eqref{eq:ODE2}. As the notation indicates, the optimal feedback policy depends on the gradient of a function $v$ called the value function (VF). It gives the optimal costs-to-go depending on an initial state. Further, if it is available, the problem of optimal feedback control is solved. From Bellman's Dynamic Programming Principle (DPP) it can be deduced that the VF is the solution of the PDE called the Hamilton-Jacobi-Bellman (HJB) equation. However, the problem with solving the HJB equation is that classical numerical solution algorithms for PDEs, such as finite differences or finite elements, suffer from the curse of dimensionality \cite{bellman1961}, making such approaches infeasible for dimensions $N>6$. Since the ODE in the OCP often originates from a semi-discretized PDE and therefore has a high state dimension, much research has been pursued in recent years to find strategies to overcome this difficulty.\\
\newline
In many methods, a semi-Lagrangian scheme  \cite{falcone2013}  is applied to the HJB equation. Advanced techniques such as spatially adaptive sparse grid \cite{bokanowski2013} are used to avoid an impractical  full grid. In the finite horizon case, using a tree-structure \cite{alla2019A,alla2020}, a reasonable discrete domain can be constructed utilizing the underlying dynamics. This approach can also be combined with model order reduction (MOR)  \cite{alla2019B}. In the infinite horizon case, from the semi-Lagrangian point of view, but also when \eqref{eq:MP}--\eqref{eq:ODE2} is discretized in time, an iterative fixed point scheme called value iteration (VI) \cite{bellman1957B} can be motivated to solve the HJB. Convergence is guaranteed for this method  \cite{falcone1987}, but it comes at a high computational cost and slow convergence, since the contraction constant tends to $1$ as the mesh becomes finer. Most methods based on the VI differ in their ansatz to the approximant of the VF. In \cite{alla2021}, a Shepard's moving least squares approximation method using radial basis functions is utilized  to generate an interpolant. A neural network can also be used for this  as in \cite{heydari2014} and the references therein. Additional references and discussion can be found in \cite{kamalapurkar2018}. \\
Another notable iterative scheme based on the HJB equation is the policy iteration (PI) \cite{bellman1957}. This requires an approximation step for the numerical solvability of the generalized linear HJB equation, which is where the methods  differ.  In \cite{kalise2018,kalise2020,dolgov2021}, this is a Galerkin projection of the residual equation using a certain polynomial basis, which is enabled by assuming separability of the data and mitigates the curse of dimensionality. A semi-Lagrangian scheme can also perform the approximation step, as in \cite{alla2020B}, where MOR then allows applicability to high-dimensional problems. For the PI, however, convergence is only guaranteed for the stabilizing initial solution of the VF \cite{saridis1979}, which makes the combination of VI and PI interesting \cite{alla2015}.\\
Furthermore, there are many techniques for generating suboptimal feedback control, such as model predictive control (MPC), which is a very relevant feedback design, see \cite{gruene2011} for an introduction. Another common feedback controller can be provided by the linearization of the problem and  the solution of the associated algebraic Riccati equation (ARE) \cite{freeman1995}. For parameter-dependent problems, this can also be combined with reduced basis methods, see \cite{schmidt2018B}, which is interesting for multi-query scenarios. However, under controllability assumptions, this only stabilizes in a certain neighborhood around the target position. A truncated Taylor expansion can also replace the VF, with the HJB equation and certain regularity assumptions leading to generalized Lyapunov equations for the higher-order terms \cite{breiten2019}. Moreover, it is also possible to solve a state-dependent Riccati equation in each step, provided that the dynamics satisfies the structure $f(x) = A(x)x$, see \cite{cimen2008} for a survey. Here, many AREs have to be solved online, which is costly. In \cite{albi2022}, this is avoided by solving many AREs for a set of training states and then fitting a neural network to this data to provide efficient online feedback control. \\
This brings us to data-driven approaches, including our method. All of the above techniques are direct methods \cite{rao2009} as they are based on the DPP. Nevertheless, there is also an indirect approach that was developed independently of Bellman and goes back to L.\ S.\ Pontryagin \cite{pontryagin1962}. There, the first-order necessary condition of the OCP leads to a two-point boundary value problem called Pontryagin's maximum principle (PMP), among whose solutions the optimal solution can be found. This procedure only generates the open-loop solution for one initial state, but with the well known relation between DPP and PMP, information about the VF and its gradient is provided along this optimal trajectory. Thus, by solving many open-loop controls for different initial states, which can be done in parallel (causal-free) in an offline phase, it is possible to generate a data set that can be used to approximate the VF. An approximate feedback control can then be generated online using a surrogate.\\
In \cite{nakamura2021}, a neural network is used and the initial states for data generation are first chosen randomly. Later, the data is enriched adaptively based on the current surrogate. Without the latter, a similar approach is taken in \cite{azmi2021}, where a hyperbolic cross polynomial model is fitted to the VF. Polynomial interpolation is also used in \cite{kang2017}, where the VF data is computed on a sparse grid. \\
\newline
The current presentation is based on the idea in \cite{schmidt2018} and extends \cite{ehring2021,ehring2022}. In our approach, the domain of interest is explored by optimal trajectories starting from a problem-dependent set of initial states. This avoids having to specify a domain of possible states, which is generally very difficult. Instead, a hypercube is often used for this purpose, which is a very crude choice, since it can contain physically meaningless states if one thinks of an underlying semi-discretized PDE. Furthermore, in our data generation method, the selected initial states are chosen adaptively with the advantage that previously computed data can be used to find new, promising initial states for solving the most informative next open-loop process. This avoids redundant data and keeps it sparse.  Hermite kernel interpolation techniques are then applied to this data set to obtain a surrogate that exploits all the available information of the VF. Here an adapted Vectorial Kernel Orthogonal Greedy Algorithm (VKOGA) \cite{WH13} is used to select the interpolation points.  Generally, kernel techniques are robust to the curse of dimensionality. Moreover, they have the advantage of being grid-free and therefore are a perfect choice for scattered data. One difficulty that arises when considering Hermite kernel interpolation is that the system matrix becomes very large, making it impossible to store or work with. We overcome this with a matrix-free strategy using a special class of kernels. Furthermore, under the (strong) assumption that the VF belongs to the reproducing kernel Hilbert space (RKHS) for the considered kernel, we show the convergence of the surrogate to the VF as well as the convergence of the surrogate-controlled solution to the optimal solution. \\
 \newline
 The paper is organized as follows. Section 2 provides a brief background on optimal control. Then, in Section 3, the Hermite kernel interpolation is introduced. Additionally, the matrix-free strategy and the construction of the surrogate that satisfies certain properties of the VF is explained. In Section 4, the convergence proofs are given. We perform some numerical test in Section 5. Finally, conclusions and an outlook are presented in Section 6. 

\section{Background on optimal control}\label{sec2}
\subsection*{Optimal closed-loop control}
\noindent The core of optimal feedback control is  the value function, mentioned in the introduction and explicitly given by
\begin{gather*}
v(x) = \min_{\mathbf{u} \in \mathcal{U}_{\infty}} \int_0^{\infty} r(\mathbf{x}(s)) + \mathbf{u}(s)^{\top}R \mathbf{u}(s) \,\text{d}s  \\ 
\text{ subject to } \dot{\mathbf{x}}(s) = f(\mathbf{x}(s))+g(\mathbf{x}(s))\mathbf{u}(s)  \text{ and }\mathbf{x}(0) = x. 
\end{gather*}
Since the initial state is a free variable of the VF, it is denoted by $ x $. Furthermore, in the following we explicitly indicate the dependence of the trajectory on the initial state $x_0$ by $\mathbf{x}(\cdot;x_0)$, if this is not obvious from the context. The VF satisfies Bellman's optimality principle,  which states that any subtrajectory to the target state of an optimal trajectory is also optimal. This is reflected by
\begin{align}\label{eq:BPO}
v(x) = \min_{\mathbf{u} \in \mathcal{U}_{\infty}} \left\lbrace \int_0^{t} r(\mathbf{x}(s;x)) + \mathbf{u}(s;x)^{\top}R \mathbf{u}(s;x) \,\text{d}s + v(\mathbf{x}(t;x)) \right\rbrace .
\end{align}
Next, a precise definition of the domain of interest is given
\begin{align*}
\mathcal{T}(\mathcal{A}) :=  \closed \{&\mathbf{x}^*(t;x_0) \in \mathbb{R}^N \,\vert\, (\mathbf{x}^*(t;x_0),\mathbf{u}^*(t;x_0))  \\ &\text{solves \eqref{eq:MP}--\eqref{eq:ODE2}} \,\, \forall x_0 \in \mathcal{A} \text{ and } \forall t \in [0,\infty) \},
\end{align*}
 which is the closed set of the states of all optimal trajectories with the initial state in the compact set $\mathcal{A}$. Here $\closed$ denotes the closure. The main reasons why we can mitigate the curse of dimensionality with our data-based approach is that we are only interested in feedback control on $\mathcal{T}(\mathcal{A})$. The latter is possibly of a much lower dimension, since it contains only states that can occur under optimal control. The main assumption of the present study is that the OCP has a VF that fulfills the following assumptions:
\begin{assumption}\label{ass:VFC2}
The VF $v$ satisfies $v \in C^1(\mathcal{T}\left(\mathcal{A})\right)$ and is radially unbounded, i.e. 
\begin{align*}
v(x) \longrightarrow \infty \, \text{ for all } \, \norm{x}_2 \longrightarrow \infty.
\end{align*}
\end{assumption}
\noindent A comprehensive discussion and conditions for regularity of the VF can be found in \cite{B08,Benveniste1979}.
Assumption \ref{ass:VFC2} makes the set $\mathcal{T}(\mathcal{A}) $ compact:
\begin{lemma}\label{lemma:compact}
Let $\mathcal{A}\subset \mathbb{R}^N$ be a non-empty compact set and the VF $v$ satisfy Assumption \ref{ass:VFC2}, then the set $\mathcal{T}(\mathcal{A})$ is bounded.
\end{lemma}
\begin{proof}
We prove this by contradiction. Assume $\mathcal{T}(\mathcal{A})$ is unbounded, and so there exists a sequence  $\left(t_k,x_{0,k}\right)_{k \in \mathbb{N}} \subset [0,\infty) \times \mathcal{A}$ such that 
\begin{align*}
\norm{\mathbf{x}^*(t_k;x_{0,k})}_{2} \longrightarrow \infty.
\end{align*}
 Since the set $ \mathcal{A} $ is compact, the continuous VF assumes its maximum on it, i.e. there is a $M\in\mathbb{R}_+$ with 
\begin{align*}
M \geq v(x) \,\text{ for all } x\in \mathcal{A}.
\end{align*}
Then Bellman's principle of optimality leads to
\begin{align*}
M \geq v(\mathbf{x}^*(0;x_{0,k})) &=  \int_0^{t_k} r(\mathbf{x}^*(s;x_{0,k}),\mathbf{u}^*(s;x_{0,k})) \,\text{d}s + v(\mathbf{x}^*(t_k;x_{0,k})) \\ &\geq v(\mathbf{x}^*(t_k;x_{0,k})),
\end{align*}
which contradicts the assumption that $v(\mathbf{x}^*(t_k;x_{0,k})) \longrightarrow \infty$ for $k \longrightarrow \infty $, as $\norm{\mathbf{x}^*(t_k;x_{0,k})}_{2} \longrightarrow \infty$ for $k \longrightarrow \infty $.
\end{proof}
\noindent Note that the statement is already implied by  $v \in C(\mathcal{A})$. However, for the sake of clarity, we have combined further assumptions needed below into Assumption \ref{ass:VFC2}. A characterization of the VF as a solution of a PDE can be obtained via  \eqref{eq:BPO} by rearranging, dividing by $t$ and letting $t$ go to 0. It follows
\begin{align}
 \min_{u \in \mathbb{R}^M}\left\lbrace \left(f(x)+g(x)u\right)\cdot \nabla_x v(x) + r(x) + u^{\top}R u \right\rbrace = 0 \label{eq:HBJ}
\end{align}
for at least all $x \in \mathcal{T}(\mathcal{A})$, which is the famous HJB equation. Due to the structure of the considered problem, the local minimization in \eqref{eq:HBJ} can be solved exactly. This yields
\begin{align*}
\nabla v(x)^{\top} f(x) - \frac{1}{4} \nabla v(x)^{\top} g(x) R^{-1} g(x)^{\top}\nabla v(x)+r(x) = 0,
\end{align*}
where the minimizer is
\begin{align} 
\mathcal{K}(x;\nabla v) := u(x) = -\frac{1}{2}R^{-1}g(x)^{\top} \nabla v(x). \label{eq:feedback}
\end{align}
The optimal feedback rule is also given by the latter equation. So if  $\nabla v$ is known,  it can be used to obtain the optimal feedback policy. However, solving the HJB is not possible in higher dimensions as already mentioned. Therefore, in our approach, a data-based surrogate for the VF is constructed using the available data from optimal open-loop controlled trajectories.
\subsection*{Optimal open-loop control}
As mentioned in the introduction, there are two conceptually different ways of solving an open-loop control problem. The first is the class of direct methods, where the problem is transformed into a non-linear programming problem by discretization. This can be seen as "first discretize, then optimize".   However, this does not provide any additional information about the system, in particular about the gradient of the VF. This requires indirect methods using PMP-like conditions, which can be understood as "first optimize, then discretize". Since we are interested in this additional information for the Hermite interpolation, we present the indirect ansatz below. For what follows, we need another mild assumption. This is a summary and adaptation of Assumptions (A0) and (A1) in \cite{aseev2019}.
\begin{assumption}\label{ass:PMP}
Let $r,f,g\in C^1(\mathbb{R}^N)$  and let further for all $x_0 \in \mathcal{A}$ exist a continuous function $\varphi_1 : [0,1) \longrightarrow (0,1)$ and a locally integrable function $\varphi_2 : [0,1) \longrightarrow  \mathbb{R}$, such that for almost every $t \geq 0$ it holds
\begin{align*}
\max_{ \bar{x} \, : \, \norm{\bar{x}-\mathbf{x}^*(t)}_2 \leq \varphi_1(t) } \left\{ \sum_{i=1}^N\norm{ \nabla_x \left( f(\bar{x})+g(\bar{x})\mathbf{u}^*(t) \right)_i }_2 + \norm{ \nabla_x r(\bar{x})}_2 \right\} \leq \varphi_2(t)
\end{align*}
with  $(\mathbf{x}^*,\mathbf{u}^*)$ being the optimal process corresponding to $x_0$.
\end{assumption}
\noindent This makes it possible to specify the necessary conditions that an optimal process must fulfill:
\begin{theorem}
Let Assumption \ref{ass:PMP} be satisfied and let further $(\mathbf{x}^*,\mathbf{u}^*)$ be an optimal process, then there is a $q_0 \geq 0$ and locally absolutely continuous co-state $\mathbf{p}^* :[0,\infty) \longrightarrow \mathbb{R}^N$ such that
\begin{align}
\dot{\mathbf{x}}^*(s) &= \nabla_p H\left(\mathbf{x}^*(s),q_0 ,\mathbf{p}^*(s),\mathbf{u}^*(s) \right), \label{eq:PMP1} \;\;\;\;\;\;\;\;\,
\mathbf{x}^*(0) = x,   \\ 
\dot{\mathbf{p}}^*(s) &= -\nabla_x H\left(\mathbf{x}^*(s),q_0 ,\mathbf{p}^*(s),\mathbf{u}^*(s) \right),\label{eq:PMP2} \; \;\;\;\;
 \\
\mathbf{u}^*(s) &=\argmin_{u \in \mathbb{R}^M}\left\lbrace H\left(\mathbf{x}^*(s),q_0 ,\mathbf{p}^*(s),u \right) \right\rbrace \label{eq:PMP3} & & 
\end{align}
are valid for $s \in [0,\infty)$, where $H(x,q,p,u) = (f(x)+g(x)u)\,\cdot\,p +  q (r(x) + u^{\top}R u)$ is the control Hamiltonian.
\end{theorem}
\noindent The proof for this theorem can be found in \cite{aseev2019}. These equations are the core conditions of the PMP, since they determine the dynamics of the system. We extend them by an additional transversality condition  
\begin{align}\label{eq:transInfinit}
\lim_{t \longrightarrow \infty} \mathbf{p}^*(t) = 0.
\end{align}
to further restrict the solution space. It is motivated by the finite horizon OCP, where the PMP equations are given by the upper core condition \eqref{eq:PMP1}--\eqref{eq:PMP3} with the additional transversality condition $\mathbf{p}^*_T(T) = 0$. Here,  $T>0$ is the finite time horizon. However, under the current assumption, \eqref{eq:transInfinit} need not be satisfied as $T$ goes to infinity, as the famous Halkin counterexample shows. Therefore we pose another assumption and refer the reader to \cite{aseev2019,michel1982} where additional conditions implying \eqref{eq:transInfinit} are given.
\begin{assumption}\label{ass:trans}
Let the co-state satisfy \eqref{eq:transInfinit} and let $q_0=1$.
\end{assumption}
\noindent In the latter we also include the assumption that we do not consider abnormal cases, i.e. $q_0=0$. In all other cases, we have, without loss of generality, that $q_0=1$ (see \cite{aseev2019} for instance). Overall, under Assumptions \ref{ass:PMP} and \ref{ass:trans}, the PMP conditions for the infinite horizon OCP are given by
\begin{align}
\dot{\mathbf{x}}^*(s) &= \nabla_p H\left(\mathbf{x}^*(s),\mathbf{p}^*(s),\mathbf{u}^*(s) \right), \;\;\;\;\;\;\;\;\,
\mathbf{x}^*(0) = x,  \label{eq:PMP-1} \\ 
\dot{\mathbf{p}}^*(s) &= -\nabla_x H\left(\mathbf{x}^*(s),\mathbf{p}^*(s),\mathbf{u}^*(s) \right), \; \;\;\;\;
\lim_{t \longrightarrow \infty} \mathbf{p}^*(t) =: \mathbf{p}^*(\infty) = 0, \label{eq:PMP-2} \\
\mathbf{u}^*(s) &=\argmin_{u \in \mathbb{R}^M}\left\lbrace H\left(\mathbf{x}^*(s),\mathbf{p}^*(s),u \right) \right\rbrace & & \label{eq:PMP-3}
\end{align} 
for $s \in [0,\infty)$ and with the abbreviation $H(x,p,u) := H(x,q=1,p,u) =  (f(x)+g(x)u)\,\cdot\,p +  r(x) + u^{\top}R u$. With respect to our data-based Hermite interpolation approach, the next well-known relationship (see \cite{B08,aseev2019,michel1982}) between the co-state and the gradient of the VF is very important:
\begin{align} \label{eq:coStateDVF}
\mathbf{p}^*(s) =  \nabla_x v(\mathbf{x}^*(s)) 
\end{align}
It is valid with Assumption  \ref{ass:VFC2}. One can also determine the values of the VF along the optimal trajectories by an additional function $\mathbf{v}^*:[0,\infty) \longrightarrow \mathbb{R}_+$ which satisfies
\begin{align}\label{eq:PMP-4}
\dot{\mathbf{v}}^*(s) = - r(\mathbf{x}^*(s)) - \mathbf{u}^*(s)^{\top}R \mathbf{u}^*(s) \; \text{ with } \; \lim_{t \longrightarrow \infty} \mathbf{v}^*(t) =: \mathbf{v}^*(\infty) = 0. 
\end{align}
Note that $ \mathbf{v}^*(t) = v(\mathbf{x}^*(t))$. The conditions in \eqref{eq:PMP-1}--\eqref{eq:PMP-3} and  \eqref{eq:PMP-4} can be formulated as an inhomogeneous Dirichlet two-point boundary value problem, among whose solutions the optimal one can be found. The vector 
\begin{align*}
 \mathbf{z}^*(s) := \left[\begin{array}{ccc} \mathbf{x}^*(s)^{\top} & \mathbf{p}^*(s)^{\top} & \mathbf{v}^*(s)  \end{array} \right]^{\top}
\end{align*}
is introduced for this purpose and  the right-hand side  
\begin{align*}
F(\mathbf{z}^*) :=  \left[\begin{array}{c} \nabla_p H\left(\mathbf{x}^*,\mathbf{p}^*,\mathbf{u}^*(\mathbf{x}^*,\mathbf{p}^*) \right)\\[2pt]  -\nabla_x H\left(\mathbf{x}^*,\mathbf{p}^*,\mathbf{u}^*(\mathbf{x}^*,\mathbf{p}^*) \right)\\[2pt]
- r(\mathbf{x}^*) - \left( \mathbf{u}^*(\mathbf{x}^*,\mathbf{p}^*) \right)^{\top}R \; \left(\mathbf{u}^*(\mathbf{x}^*,\mathbf{p}^*) \right) \end{array} \right]
\end{align*}
with $\mathbf{u}^*(\mathbf{x}^*(s),\mathbf{p}^*(s)):= \argmin_{u \in \mathbb{R}^M}\left\lbrace H\left(\mathbf{x}^*(s),\mathbf{p}^*(s),u \right) \right\rbrace $. The boundary conditions in \eqref{eq:PMP-1}, \eqref{eq:PMP-2} and \eqref{eq:PMP-4}  can be compactly represented by
\begin{align*}
b(\mathbf{z}^*(0),\mathbf{z}^*(\infty)) := \left[\begin{array}{cc} I_N & 0 \\ 0 & 0 \end{array}\right]\mathbf{z}^*(0)+\left[\begin{array}{cc} 0 & 0 \\ 0 & I_{N+1} \end{array}\right]\mathbf{z}^*(\infty) - \left[\begin{array}{c} x_0\\[2pt] 0 \end{array}\right] =0 .
\end{align*}
Here, $I_N \in \mathbb{R}^{N \times N}$ is the $N \times N  $ identity matrix. In total,
\begin{align}\label{eq:compatPMP1}
\dot{\mathbf{z}}^*(s) = F(\mathbf{z}^*(s)) \, \text{ with }\,
b(\mathbf{z}^*(0),\mathbf{z}^*(\infty)) = 0,
\end{align}
needs to be solved, where the infinity part of the boundary condition has to be understood in the limit sense. For example, with the basic idea of \cite{fahroo2008}, such problems with an infinite time horizon can be treated numerically. That article also presents  references for theory about this type of algorithm, which we omit here. In general, the technique is to map the interval $[0,1)$ to the interval $[0,\infty)$ using a function such as $\Phi(\tau) = \frac{\tau}{1-\tau}$. The transformed system $\bar{\mathbf{z}}^*(\tau):=\mathbf{z}(\Phi(\tau))$ on the finite interval $[0,1)$ with the new time variable $\tau$ is then
\begin{align}\label{eq:compatPMP2}
\dot{\bar{\mathbf{z}}}^*(\tau) = \Phi'(\tau) F(\bar{\mathbf{z}}^*(\tau)) \, \text{ with }\,
b(\bar{\mathbf{z}}^*(0),\bar{\mathbf{z}}^*(1)) = 0.
\end{align}
However, even for this finite time system, solvers for inhomogeneous Dirichlet two-point boundary value problems such as \emph{bvp5c} from Matlab or \emph{solve{\_}bvp} from SciPy cannot be used. These evaluate the right-hand side of \eqref{eq:compatPMP2} at the singularity $\tau = 1$.  Similar to \cite{garg2011}, this can be avoided by using one step of an explicit Euler method. An approximation to the solution  $\bar{\mathbf{z}}^*(1)$ on the right boundary  can be generated as follows:
\begin{align}
\bar{\mathbf{z}}^*(1) \approx \bar{\mathbf{z}}^*(1-\Delta\tau) + \Delta\tau \Phi'(1-\Delta\tau) F(\bar{\mathbf{z}}^*(1-\Delta\tau)).
\end{align}
The system which is now on the finite time interval $[0,1-\Delta\tau]$ without singularity problems and approximates \eqref{eq:compatPMP1} is
\begin{align}
&\dot{\bar{\mathbf{z}}}^*(\tau) = \Phi'(\tau) F(\bar{\mathbf{z}}^*(\tau)) \label{eq:compatPMP21}\\ \text{ with }\,
&b(\bar{\mathbf{z}}^*(0), \bar{\mathbf{z}}^*(1-\Delta\tau) + \Delta\tau \Phi'(1-\Delta\tau) F(\bar{\mathbf{z}}^*(1-\Delta\tau))) = 0.\label{eq:compatPMP22}
\end{align}
Overall solving \eqref{eq:compatPMP21}--\eqref{eq:compatPMP22} allows information to be gathered about both the VF and its gradient along optimal trajectories. At this point, however, it is not yet clear which optimal trajectories are chosen. This is realized using the greedy exploration algorithm that we introduced in \cite{ehring2022} which builds a data  set $\mathcal{D} \subset \mathcal{T}(\mathcal{A})$ that approximates $\mathcal{T}(\mathcal{A})$. Subsequently, on $\mathcal{D}$ a feedback rule is  generated. The idea of this algorithm is to compute the optimal open-loop trajectories only for particularly promising initial states of $\mathcal{A}$ in order to expand the entire data set. It is given by the Algorithm \ref{algo:greedy1}. 
\begin{algorithm}
\caption{Greedy exploration algorithm}\label{algo:greedy1}
\begin{algorithmic}[1]
\State Input:  $\mathcal{A}$, $\mathcal{X}:=\emptyset$, $\mathcal{D}= \{0\}$ and  $\epsilon_{\text{tol,d}}$
\State $\epsilon := \epsilon_{\text{tol,d}}+1$
\State while {$\epsilon>\epsilon_{\text{tol,d}}$}
       \State \qquad $x_a  := \argmax_{x\in \mathcal{A}}  \min_{y \in \mathcal{D} } \norm{x - y}_{2}$ \label{algoStep:greedy13}
       \State \qquad $\epsilon :=   \min_{y \in \mathcal{D}} \norm{x_a - y}_{2} $ \label{algoStep:greedy14}
       \State \qquad $[\mathbf{x}(\, \cdot \,;x_a),\,v(\mathbf{x}(\, \cdot \,;x_a)),\,\nabla v(\mathbf{x}(\, \cdot \,;x_a))] := \text{solveOpenLoop}(x_a)$ \label{algoStep:greedy15}
       \State  \qquad $\mathcal{D} := \mathcal{D} \cup \{\mathbf{x}(\,\cdot\,;x_a) \}$     \label{algoStep:greedy16}
\end{algorithmic}
\end{algorithm}
The rule in step \ref{algoStep:greedy13} is used to find new starting positions $x_a$. The solveOpenLoop function in step \ref{algoStep:greedy15} outputs the optimal trajectory and the corresponding value of the VF and its gradient for the initial state $x_a$. The data set $\mathcal{D}$ is then expanded by that trajectory in step \ref{algoStep:greedy16}.  This process is performed until the cover distance in step \ref{algoStep:greedy14} is smaller than a given tolerance $\epsilon_{\text{tol,d}}$. Here, the quantity $\epsilon$ behaves like $n^{-\frac{1}{N}}$ in the worst case, where $n$ is the number of iterations. This is due to the fact that the fill distance is an upper bound of the cover distance \cite{ehring2022,MSW05}. The outputs of this algorithm are the sets $\mathcal{D}, v(\mathcal{D})$ and $\nabla v(\mathcal{D})$. 
\section{On Hermite kernel surrogates}
We recall some basic information about kernels. A symmetric function $k\,:\, \Omega \times \Omega \longrightarrow \mathbb{R}$ for a non-empty set $\Omega$ is called a kernel. If for each finite pairwise distinct set $X_n :=\{x_1,...,x_n\} \subset \Omega$, the Gram matrix $ (\mathcal{K}_{X_n})_{i,j} = k(x_i,x_j)$ is positive semidefinite, then the kernel is called positive definite (p.d.) and if all such Gram matrices are positive definite, then the kernel is called strictly positive definite (s.p.d.). Obviously, all s.p.d kernels are also p.d kernels, the latter being of particular interest since each of them can be associated with a unique reproducing kernel Hilbert space $\mathcal{H}_k(\Omega)$ (RKHS), where $ k $ is the corresponding reproducing kernel. A RKHS is a Hilbert  space of functions $f:\Omega \longrightarrow \mathbb{R}$ with the property that there exists a kernel $k\,:\, \Omega \times \Omega \longrightarrow \mathbb{R}$, such that $k(x,\,\cdot\,) \in \mathcal{H}_k(\Omega)$ for all $x \in \Omega$ and $\langle f , k(x,\,\cdot\,) \rangle_{\mathcal{H}_k(\Omega)} = f(x)$ for all $f \in \mathcal{H}_k(\Omega)$. The latter is called reproducing property. An introduction can be found in \cite[Chapter 10]{wendland2004}.

\subsection*{Hermite kernel interpolation}
We continue with the introduction of generalized kernel interpolation as discussed in \cite[Chapter 16]{wendland2004}, which also covers the Hermite kernel interpolation. For a p.d. kernel, the interpolation conditions for an unknown interpolant $s_f^n \in \mathcal{H}_k(\Omega)$ in a generalized setting are given by
\begin{align}\label{eq:genInterpol}
\lambda_i (s_f^n) = f_i,
\end{align}
where $\lambda_1,...,\lambda_n \in \mathcal{H}_k(\Omega)'$ are linear functionals contained in the dual space of the RKHS and $f_1,...,f_n \in  \mathbb{R}$ are some target values. From Theorem 16.1 in \cite{wendland2004},  it follows that \eqref{eq:genInterpol} can be solved uniquely if the $\lambda_1,...,\lambda_n \in \mathcal{H}_k(\Omega)'$ are linearly independent. Furthermore the interpolant is given by
\begin{align*}
s_f^n = \sum_{i=1}^n \alpha_i v_i,
\end{align*}
where $v_i \in \mathcal{H}_k(\Omega) $ is the Riesz representer of $\lambda_i$ and the coefficients $\alpha_1,...,\alpha_n \in \mathbb{R}$ are determined via the interpolation condition in \eqref{eq:genInterpol}. The considered Hermite interpolation problem is of the following type: It is assumed that for a set $\mathcal{D} \subset \Omega$, the values of the VF $v(\mathcal{D})$ and of its gradient $\nabla v(\mathcal{D})$ are given. For a p.d. kernel $k \in C^2(\Omega \times \Omega )$ and a finite pairwise distinct set $X_n = \{x_1,...,x_n \} \subset \mathcal{D}$, the goal is to find a surrogate such that
\begin{align}\label{eq:interDeriv}
\partial^{a} s_{v}^n(x_i) = \delta_{x_i} \partial^{a} s_{v}^n = \delta_{x_i} \partial^{a} v  = \partial^{a} v(x_i)
\end{align}
for all multiindices $ a \in \mathbb{N}^N_0 \text{ with } \sum_{l=1}^N (a)_l \leq 1$ and $i=1,...,n$. Note that both the point evaluation functional $\delta_{x_i}$ and the combination with the partial derivative $\delta_{x_i} \partial^{a}$ are in the dual space $\mathcal{H}_k(\Omega)'$ due to the reproducing property. The Riesz representers of these functionals are given by  $ \partial^{a}_1 k(x_i,\cdot) $ since $\delta_{x_i} \partial^{a} g = \langle \partial^{a}_1 k(x_i,\cdot) ,g \rangle_{\mathcal{H}_k(\Omega)}$ for all $g \in \mathcal{H}_k(\Omega)$ (see Theorem 10.45 in \cite{wendland2004}). Here the $1$ in the subscript of $ \partial^{a}_1$ denotes that it acts on the first input variable of the kernel.  Overall, the surrogate can  be represented by
\begin{align}\label{eq:surrogateForm}
s_{v}^n(x) =  \sum_{i=1}^n \alpha_i k(x_i,x) + \langle \beta_i , \nabla_1 k(x_i,x) \rangle .
\end{align}
The condition from \eqref{eq:interDeriv} for determining the coefficients $\{\alpha_i\}_{i=1}^n \subset \mathbb{R}$ and $\{\beta_i\}_{i=1}^n \subset \mathbb{R}^N$ can  be written compactly with matrices as follows 
\begin{align*}
\begin{bmatrix}
\mathcal{K}_{X_n} & \mathcal{B}_{X_n} 
\end{bmatrix}
\begin{bmatrix}
\underline{\alpha} \\ \underline{\beta} 
\end{bmatrix} = \begin{bmatrix}
v(x_1) & \hdots & v(x_n)
\end{bmatrix}^{\top} =: \underline{v},
\end{align*}
where
\begin{gather*}
\mathcal{K}_{X_n} := \begin{bmatrix}
k(x_1,x_1) & \hdots & k(x_n,x_1) \\ 
\vdots & \ddots & \vdots  \\ 
k(x_1,x_n) & \hdots & k(x_n,x_n) 
\end{bmatrix} \in \mathbb{R}^{n \times n},\;
\mathcal{B}_{X_n} := \begin{bmatrix}
\nabla_1 k(x_1,x_1)^{\top}  & \hdots & \nabla_1 k(x_n,x_1)^{\top}  \\ 
\vdots & \ddots & \vdots  \\ 
\nabla_1 k(x_1,x_n)^{\top}  & \hdots & \nabla_1 k(x_n,x_n)^{\top}  
\end{bmatrix}\in \mathbb{R}^{n \times nN}, \\
\underline{\alpha} := \begin{bmatrix}
\alpha_1 & 
\hdots   &
\alpha_n 
\end{bmatrix}^{\top}\in \mathbb{R}^{n} \text{ and } \underline{\beta} := \begin{bmatrix}
\beta_1^{\top} & 
\hdots   &
\beta_n^{\top} 
\end{bmatrix}^{\top} \in \mathbb{R}^{nN}
\end{gather*}
for the part of the conditions involving the values of the VF. For the derivative part, we first take a look at the gradient of the surrogate, which is 
\begin{align*}
\nabla s_{v}^n(x) =  \sum_{i=1}^n \alpha_i \nabla_2 k(x_i,x) +  \mathcal{E}_k(x_i,x) \beta_i 
\end{align*}
with 
\begin{align*}
 \mathcal{E}_k(x_i,x) := \begin{bmatrix}
\partial_{2,1} \nabla_1 k(x_i,x)^{\top} \\ 
\vdots  \\ 
\partial_{2,N} \nabla_1 k(x_i,x)^{\top}
\end{bmatrix} \in \mathbb{R}^{N \times N}.
\end{align*}
Here $ \nabla_2 $ denotes the gradient that acts on the second input variable of the kernel and $\partial_{2,i}$ the partial derivative that acts on the $ i $-th component of the second input variable. A compact matrix notation of the second part  then results in
\begin{align*}
\begin{bmatrix}
\mathcal{B}_\mathcal{X}^{\top} & \mathcal{C}_\mathcal{X} 
\end{bmatrix}
\begin{bmatrix}
\underline{\alpha} \\ \underline{\beta} 
\end{bmatrix} = \begin{bmatrix}
\nabla v(x_1)^{\top} & \hdots & \nabla v(x_n)^{\top}
\end{bmatrix}^{\top} =: \underline{\nabla v},
\end{align*}
where
\begin{align*}
\mathcal{C}_\mathcal{X} = \begin{bmatrix}
\mathcal{E}_k(x_1,x_1) & \hdots & \mathcal{E}_k(x_n,x_1) \\ 
\vdots & \ddots & \vdots  \\ 
\mathcal{E}_k(x_1,x_n) & \hdots & \mathcal{E}_k(x_n,x_n) 
\end{bmatrix}\in \mathbb{R}^{nN \times nN}.
\end{align*}
In doing so, we have taken advantage of the fact that $\nabla_2 k(x_i,x_j) = \nabla_1 k(x_j,x_i)$. Ultimately, the system of linear equations for determining the coefficients for the surrogate is
\begin{align}\label{eq:matrixM}
\underbrace{\begin{bmatrix}
\mathcal{K}_{X_n} & \mathcal{B}_{X_n} \\
\mathcal{B}^{\top}_{X_n} & \mathcal{C}_{X_n}
\end{bmatrix}}_{=:\mathcal{M}_{X_n}} \begin{bmatrix}
\underline{\alpha} \\ \underline{\beta} 
\end{bmatrix} = \begin{bmatrix}
\underline{v} \\ \underline{\nabla v}
\end{bmatrix}.
\end{align}
\noindent 
By construction, the matrix $\mathcal{M}_{X_n} \in \mathbb{R}^{N(n+1) \times N(n+1)}$ is symmetric. Furthermore, it can be shown that it is positive definite for a suitable choice of kernel and therefore the system is uniquely solvable.
\begin{proposition}\label{theo:regM}
Let $k(x,y)= \Phi(x-y)$ be a s.p.d kernel with  $\Phi\in C^{2}(\Omega)\cap L^1(\Omega)$, then the matrix $\mathcal{M}_{X_n}$  is symmetric positive definite for all pairwise distinct $X_n \subset \mathbb{R}^N$.
\end{proposition}
\begin{proof}
First, we recall Theorem 16.4 in \cite{wendland2004}, which states that under the assumed condition the set
\begin{align*}
\{ \delta_x \partial^{a} \,\vert\, x\in \Omega \text{ and } (a_i)_{i=1}^N \in \mathbb{N}_0^N \text{ with } \sum_{i=1}^N a_i \leq m \} \subset \mathcal{H}_k(\Omega)'
\end{align*}
is linearly independent. So, since
\begin{align*}
\partial_{1,i} k(y,x) = \left\langle \partial_{1,i} k(y,\,\cdot\,) , k(x,\,\cdot\,) \right\rangle_{\mathcal{H}_k(\Omega)}
\end{align*}
and
\begin{align*}
\partial_{2,j}\partial_{1,i} k(y,x) = \left\langle \partial_{1,i} k(y,\,\cdot\,) , \partial_{1,j} k(x,\,\cdot\,) \right\rangle_{\mathcal{H}_k(\Omega)}
\end{align*}
for all $\alpha \in \mathbb{R}^{N(n+1)} \setminus \{0 \}$,  it follows that
\begin{align*}
{a}^{\top} \mathcal{M}_{X_n} a &= \norm{ \sum_{i=1}^n a_i k(x_i, \,\cdot \,)+ \sum_{i=1}^n \sum_{j=1}^N a_{(iN+j)}\, \partial_{1,j} k(x_i,\,\cdot\,)  }_{\mathcal{H}_k(\Omega)}^2  \\
&= \norm{ \sum_{i=1}^n a_i \delta_{x_i}+ \sum_{i=1}^n \sum_{j=1}^N a_{(iN+j)}\, \delta_{x_i}\partial_{1,j}   }_{\mathcal{H}_k(\Omega)'}^2 \neq 0, 
\end{align*}
which shows that the matrix is positive definite.
\end{proof}
\noindent Regarding the matrix $\mathcal{M}_{X_n}\in\mathbb{R}^{N(n+1)\times N(n+1)}$, there is a computational problem since it scales quadratically in $ n$, making it impossible to construct for a set $X_n$ with many centers and a high state dimension $N$. However, as the next section will show, this is not necessary.
\subsection*{Matrix-free approach}\label{sec:matrixFree}
An explicit representation of the symmetric positive definite matrix $\mathcal{M}_{X_n}$ is not required if an iterative scheme like the CG-method is used to find the solution of the linear system of equations. All that is required here is the implementation of a matrix-vector multiplication. By restricting the consideration to a special class of radial kernels, the matrix-vector multiplication becomes efficient. The kernels in this class are of the form $k(x,y) = \Phi( \norm{x-y}^2)$ with $\Phi\in C^{2}(\mathbb{R}_0^+,\mathbb{R})$.  To see that this choice is beneficial, we look at  $\mathcal{E}_k(x,y)$:
\begin{align*}
\mathcal{E}_k(x,y) &= -2\Phi'( \norm{x-y}_2^2) \mathbb{I}_{N \times N } + 4\Phi''( \norm{x-y}_2^2)(x-y) (y-x)^{\top}.
\end{align*}
Thus the matrix $\mathcal{E}_k(x,y)$ consists of two scalar factors, a  unit matrix $\mathbb{I}_{N \times N }\in \mathbb{R}^{N \times N}$ and the rank-one matrix $(x-y) (y-x)^{\top}$. Both the unit matrix and the latter matrix are advantageous in terms of matrix-vector multiplication, since
\begin{align*}
\mathcal{C}_{X_n} \underline{\beta}  \!= \!\begin{bmatrix}
 \sum_{i=1}^n \!-2\Phi'( \norm{x_i\!-\!x_1}_2^2) \beta_i \!+ \!4\Phi''( \norm{x_i\!-\!x_1}_2^2)\!(x_i\!-\!x_1)\! \langle x_1\!-\!x_i, \beta_i \rangle_{\mathbb{R}^N}
\\ \vdots \\ 
 \sum_{i=1}^n\! -2\Phi'( \norm{x_i\!-\!x_n}_2^2) \beta_i \!+ \!4\Phi''( \norm{x_i\!-\!x_n}_2^2)\!(x_i\!-\!x_n)\! \langle x_n\!-\!x_i, \beta_i \rangle_{\mathbb{R}^N}
\end{bmatrix},
\end{align*} 
which reduces the total number of operations from  $O(N^2 n^2)$ to $O(N n^2)$.  The matrix-free approach becomes even more relevant when it comes to an efficient evaluation of the gradient of the surrogate on the input data $\mathcal{D}$. In this case, the corresponding matrix would have $\vert\mathcal{D}\,\vert N \times n N$ entries, which would be impossible to store even for small data sets. Here $ \vert\mathcal{D}\vert$ is the size of the input data. In this case, the number of operations can also be reduced from $O(N \,\vert\mathcal{D}\vert\, n^2)$ to $O(\vert\mathcal{D}\vert\, n^2)$.   The evaluation of the surrogate on given input data is necessary, for example, when using a greedy selection criterion for the centers, as introduced in the next subsection.  
\subsection*{A greedy selection criterion for the interpolation points}
At this point we have not yet clarified how the centers that represent the interpolation points  will be selected. This is determined by an iterative scheme that successively increases the set of centers through a greedy strategy. The algorithm is a version based on the  VKOGA from \cite{WH13}, which by Algorithm \ref{algo:greedy2} has been adapted to the Hermite interpolation. 
\begin{algorithm}
\caption{Hermite VKOGA}\label{algo:greedy2}
\begin{algorithmic}[1]
\State Input: $\mathcal{D}, v(\mathcal{D})$, $\nabla v(\mathcal{D})$ , $X:=\emptyset$,  $s_{v}:=0$ and  $\epsilon_{\text{tol,f}}$
\State{while $ \max_{x\in \mathcal{D} \setminus X}\left(\vert v(x)- s_{v}(x)\vert+ \norm{\nabla v(x)- \nabla s_{v}(x)}_{2} \right)> \epsilon_{\text{tol,f}}$}
       \State \qquad $x := \argmax_{x\in \mathcal{D} \setminus X} \left(\vert v(x)- s_{v}(x)\vert +  \norm{\nabla v(x)- \nabla s_{v}(x)}_{2} \right) $ \label{algoStep:greedy22}
       \State \qquad $X := X \cup \{x\}$ \label{algoStep:greedy23}
       \State \qquad $s_{v} :=  \text{interpolant}\left(X,\nabla v(X)\right) $ \label{algoStep:greedy24}
\end{algorithmic}
\end{algorithm}
The procedure is performed until the  interpolation error is less than $ \epsilon_{\text{tol,f}}$. At each iteration step in the loop, the algorithm selects the state that has the greatest deviation between the current interpolation and the output data in the sense that both the function value and the value of the gradient are included.  The picked state is then added to the set of centers in  Step \ref{algoStep:greedy23}. Thus the resulting Hermite interpolant has no error there. The output of the algorithm is a surrogate $s_{v} \approx  v$.  In general, VKOGA is easy to implement and it often leads to very good results. Theoretical results on the observed excellent convergence rates can be found in \cite{wenzel2022}.
\subsection*{A structured surrogate for the value function}\label{sec:strucHermite}
Now, we  describe a strategy for structuring the surrogate of the VF in a way that incorporates three known properties of the VF. This makes the surrogate more context-aware. The first property is the fact that $v(0) = 0$ and $\nabla v(0) = 0$, since the zero state is the target state and therefore we have no cost there. By using a kernel $k$ which satisfies $k(x,0)=0$, $\partial_{1,i} k(x,0) = 0$ and $\partial_{2,j} \partial_{1,i} k(x,0) = 0$ for all $x \in \Omega$ and all $i,j=1,...,n$, one can ensure that the surrogate also fulfills this. Such a kernel can be generated, for example, by multiplying any kernel $k'$ by $\langle x,y \rangle^2$. If $k'$ is p.d., then so is the resulting kernel. Furthermore, we then remove the state $0$ from the data set $\mathcal{D}$. Another possibility would be to force $s_v^n(0)=0$ via an interpolation condition, but this would be incompatible with the realization of the next property. The second property to consider is that the VF can be very well represented locally near zero as a quadratic function. For example, if $\left( J_f(0),g(0)\right)$ is controllable (see \cite{sontag2013}), it can be determined by the linearized OCP. Here  $J_f$ is the Jacobian matrix of $f$. The VF of the linearized problem can be generated by solving the corresponding algebraic Riccati equation. It has the form $v_{\text{local}}(x) = x^{\top} Q x$, where $Q\in \mathbb{R}^{n \times n}$ is a positive definite matrix. This is added to the surrogate so that only a correction term is determined by the data-based strategy. The third property that the surrogate should adapt is positivity, since the costs considered must always be positive. This is achieved by squaring. Altogether, the structured surrogate is given by
\begin{align*}
s_{\text{Str},v}^n(x) = \left( \sqrt{ x^{\top} Q x} + \sum_{i=1}^n \alpha_i k(x_i,x) + \langle \beta_i , \nabla_1 k(x_i,x) \rangle \right)^2.
\end{align*}
As in the penultimate subsection, the interpolation condition for determining the coefficients $\{\alpha_i\}_{i=1}^n \subset \mathbb{R}$ and $\{\beta_i\}_{i=1}^n \subset \mathbb{R}^N$ are 
\begin{align} \label{eq:inter}
s_{\text{Str},v}^n(x_j) = v(x_j) \, \text{ and } \nabla s_{\text{Str},v}^n(x_j) = \nabla v(x_j)  \text{ for all } j =1,...,n.
\end{align}
These conditions are not linear with respect to $\{\alpha_i\}_{i=1}^n$ and $\{\beta_i\}_{i=1}^n$, but linear conditions implying them can be found.  Since the VF is always positive,  
\begin{align} \label{eq:inter1}
\sum_{i=1}^n \alpha_i k(x_i,x_j) + \langle \beta_i , \nabla_1 k(x_i,x_j) \rangle \!=\! \sqrt{v(x_j)} -  \sqrt{x_j^{\top} Q x_j}  \text{ for all } j =1,...,n
\end{align}
implies the first equation in \eqref{eq:inter}. Furthermore, this linear condition can be written compactly in matrix form as follows 
\begin{align*}
\begin{bmatrix}
\mathcal{K}_{X_n} & \mathcal{B}_{X_n} 
\end{bmatrix}
\begin{bmatrix}
\underline{\alpha} \\ \underline{\beta} 
\end{bmatrix} = \begin{bmatrix}
\sqrt{v(x_1)}- \sqrt{x_1^{\top} Q x_1} & \hdots & \sqrt{v(x_n)}- \sqrt{x_n^{\top} Q x_n}
\end{bmatrix}^{\top}.
\end{align*}
\noindent The second interpolation condition can also be expressed by an equivalent linear system of equations. To see this, we first need to look at the derivative of the structured surrogate:
\begin{align*}
\nabla s_{\text{Str},v}^n(x) = 2\sqrt{ s_{\text{Str},v}^n(x) }  \left( \frac{Q x}{\sqrt{x^{\top} Q x}} + \sum_{i=1}^n \alpha_i \nabla_2 k(x_i,x) +  \mathcal{E}_k(x_i,x) \beta_i \right)
\end{align*}
 and with the interpolation conditions in \eqref{eq:inter} we get 
\begin{align*}
\nabla  v(x_j) \!= \!\nabla \! s_{\text{Str},v}^n(x_j) \!\overset{\text{\eqref{eq:inter1}}}{=}\!2 \sqrt{\!v(x_j)\!}\left(\!\!\frac{Q x_j}{\sqrt{x_j^{\top} Q x_j}} \!+\!\! \sum_{i=1}^n \alpha_i \nabla_2 k(x_i,x_j)\! +\!  \mathcal{E}_k(x_i,x_j) \beta_i \!\! \right)
\end{align*}
for all $j=1,...,n$. The case $v(x_j)=0$ cannot appear, as this is equivalent to $x_j=0$ which we have excluded from the data set $\mathcal{D}$ as explained at the beginning of the current subsection. Therefore, with the upper equation,  the linear condition implying the second interpolation condition in \eqref{eq:inter} for all $x_j$ with $v(x_j) \neq 0$ becomes 
\begin{align*}
\frac{\nabla v(x_j)}{2 \sqrt{v(x_j)}} -\frac{Q x_j}{\sqrt{x_j^{\top} Q x_j}} =  \sum_{i=1}^n \alpha_i \nabla_2 k(x_i,x_j) +  \mathcal{E}_k(x_i,x_j) \beta_i  \, \text{ for all } j=1,...,n,
\end{align*}
which can be compactly formulated with matrices by
\begin{align*}
\begin{bmatrix}
\mathcal{B}_{X_n}^{\top} & \mathcal{C}_{X_n} 
\end{bmatrix}
\begin{bmatrix}
\underline{\alpha} \\ \underline{\beta} 
\end{bmatrix} = \begin{bmatrix}
\left(\frac{\nabla v(x_1)}{2 \sqrt{v(x_1)}} -\frac{Q x_1}{\sqrt{x_1^{\top} Q x_1}}\right)^{\top} & \hdots & \left( \frac{\nabla v(x_n)}{2 \sqrt{v(x_n)}} -\frac{Q x_n}{\sqrt{x_n^{\top} Q x_n}} \right)^{\top}
\end{bmatrix}^{\top}.
\end{align*}
Therefore, the system of linear equations to obtain the surrogate is similar to \eqref{eq:matrixM}, only the right-hand side is changed. It is remarkable  that all three properties can be included natively in the surrogate without any additional effort in the solution step, as long as the matrix $Q$ is known.
\section{Two convergence results for the Hermite kernel method}
\noindent We present two convergence results for the (unstructured) Hermite surrogate and for the surrogate controlled trajectories. Another assumption combining the conditions from Sections 2 and 3 is required at this point.
\begin{assumption}\label{ass:VfinRKHS}
For a compact set $\Omega \subset \mathbb{R}^N$, let $k \in C^4 \left(\conv (\Omega) \times \conv (\Omega) \right)$ be a  s.p.d. kernel. Further let the VF $v$ be radially unbounded and  $v \in \mathcal{H}_k(\Omega )$.
\end{assumption}
\noindent Note that this assumption implies $v \in C^2(\Omega )$ (also with Theorem 10.45 in \cite{wendland2004}), which is a higher regularity requirement than that of Section 2. However, this leads to a continuous dependence  of the optimal trajectories on the initial data, which is essential for the subsequent proof of convergence. For brevity, we will use the abbreviations $\norm{\cdot}_k := \norm{\cdot}_{\mathcal{H}_k(\Omega)}$ and $\langle \cdot, \cdot \rangle_k := \langle \cdot, \cdot \rangle_{\mathcal{H}_k(\Omega)}$  in the following.
\begin{lemma}\label{lemma:surrogateLipschitz}
Under the Assumption \ref{ass:VfinRKHS}, it holds
\begin{align}\label{eq:QauasiLipschitz}
\vert s_v^n(x) - s_v^n(y) \vert + \norm{\nabla s_v^n(x) - \nabla s_v^n(y)}_2 \leq L \norm{x-y}_2
\end{align}
for all $x,y \in \Omega $, where the constant $L$ is independent of the respective Hermite kernel surrogate for the VF.  
\end{lemma}
\begin{proof}
First, we need to see that the RKHS norm of a surrogate is bounded by the RKHS norm of the VF independent of $n$:
\begin{align*}
\norm{v}_{k} \norm{s_v^n }_{k} &\geq \langle v, s_v^n \rangle_{k} + \sum_{i=1}^n \alpha_i \underbrace{\left(s_v^n(x_i)-v(x_i)\right)}_{=0} + \langle \beta_i , \underbrace{\nabla s_v^n (x_i) -  \nabla v(x_i)}_{=0} \rangle_2 \\
 &= \langle v, s_v^n \rangle_{k} + \langle s_v^n - v , s_v^n \rangle_{k} = \norm{s_v^n }_{k}^2
\end{align*} 
For the first inequality, we used the Cauchy-Schwarz inequality and inserted a zero, which is possible because of the interpolation conditions. Using the reproducing property and the sum representation of the Hermite interpolant, we have reformulated the zero term. So overall we see that $\norm{v}_k \geq \norm{s_v^n }_{k}$. This property is the reason why later the constant $L$ depends only on $\norm{v}_{k}$. To show \eqref{eq:QauasiLipschitz}, we consider
\begin{align*}
\vert s_v^n(x) - s_v^n(y) \vert^2    = & \langle  s_v^n, k(x ,\,\cdot\,)-k(y ,\,\cdot\,) \rangle_{k}^2 \\
  \leq & \norm{s_v^n}_k^2 \norm{k(x ,\,\cdot\,)-k(y ,\,\cdot\,) }_{k}^2 \\
  \leq & \norm{v}_{k}^2 (k(x ,x)-k(x,y) + k(y, y)-k(y,x)) \\
   = & \norm{v}_{k}^2 \left(\nabla_2 k(x,\xi_1)(x-y) + \nabla_2 k(y, \xi_2)(y-x)\right) \\
   \leq & \norm{v}_{k}^2\norm{x-y}_2\norm{\nabla_2 k(x,\xi_1)- \nabla_2k(y, \xi_2)}_2 \\
  = & \norm{v}_{k}^2\norm{x\!-\!y}_2\norm{\nabla_2 k(x,\xi_1)\!-\! \nabla_2 k(x,\xi_2) \!+\! \nabla_2 k(x,\xi_2) \!-\! \nabla_2k(y, \xi_2)}_2 \\
  \leq & C_1  \norm{v}_{k}^2 \norm{x-y}_2 \left(\norm{x-y}_2 + \norm{\xi_1-\xi_2}_2\right)\\
   \leq & 2 C_1  \norm{v}_{k}^2 \norm{x-y}_2^2 
\end{align*}
with  $\xi_1 = (1-t_1) y + t_1 x$ for a $t_1 \in [0,1]$ and $\xi_2 = (1-t_2) x + t_2 y$ for a $t_2 \in [0,1]$. For the first inequality we utilized the Cauchy-Schwarz inequality and for the second inequality we used the first partial result of this proof. The third and fourth inequalities follow from the mean value theorem, with a zero term inserted for the latter. Hence 
\begin{align*}
C_1:= \sqrt{N} \max_{\bar{x},\bar{y} \in \conv (\Omega)} \max_{i=1,...,N} \max \{\norm{\partial_{i,1}\nabla_2 k(\bar{x},\bar{y})}_2,\norm{\partial_{i,2} \nabla_2 k(\bar{x},\bar{y})}_2 \},
\end{align*}
because  the maximum is reached since $\conv (\Omega)$ is compact and $k \in C^4(\conv (\Omega) \times \conv (\Omega)) $ with the Assumption \ref{ass:VfinRKHS}. The last inequality follows from the fact that
\begin{align*}
\norm{\xi_1\!-\!\xi_2}_2 = \norm{ (1\!-\!t_1) y + t_1 x \!-\! (1\!-\!t_2) x - t_2 y }_2 = \vert 1\!-\!t_1\!-\!t_2 \vert \norm{  y \!-\! x }_2 \leq  \norm{x\!-\!y}_2,
\end{align*}
which is a rough estimate. This shows that $s_v^n$ is Lipschitz  continuous, where the Lipschitz  constant is independent of the respective Hermite kernel surrogate for the VF. For the partial derivative, very similar steps can be performed:
\begin{align*}
&\vert\partial_{i} s_v^n(x) - \partial_{i} s_v^n(y)\vert^2  \\ = &\vert \langle s_v^n, \partial_{i,1} k(x,\cdot) - \partial_{i,1} k(y,\cdot) \rangle_k \vert^2 \\
 \leq &\norm{s_v^n}_k^2 \norm{\partial_{i,1} k(x,\cdot) - \partial_{i,1} k(y,\cdot)}_k^2 \\
\leq &\norm{v}_k^2 \left(\partial_{i,2} \partial_{i,1} k(x,x)-\partial_{i,2} \partial_{i,1} k(x,y) + \partial_{i,2}\partial_{i,1} k(y,y) - \partial_{i,2}\partial_{i,1} k(y,x) \right) \\
\leq & \norm{v}_k^2 \norm{x-y}_2 \norm{\nabla_2 \partial_{i,2} \partial_{i,1} k(x,\tilde{\xi}_1) - \nabla_2 \partial_{i,2} \partial_{i,1} k(y,\tilde{\xi}_2)}_2 \\
= & \norm{v}_k^2 \norm{x-y}_2  \norm{\nabla_2 \partial_{i,2} \partial_{i,1} k(x,\tilde{\xi}_1)\!-\!\nabla_2 \partial_{i,2} \partial_{i,1} k(x,\tilde{\xi}_2) \!+\!\nabla_2 \partial_{i,2} \partial_{i,1} k(x,\tilde{\xi}_2)\! -\! \nabla_2 \partial_{i,2} \partial_{i,1} k(y,\tilde{\xi}_2)}_2 \\
\leq & 2C_2 \norm{v}_k^2 \norm{x-y}_2^2, 
\end{align*}
where 
\begin{align*}
C_2\!:=\! \sqrt{ N}\!\!\! \max_{\bar{x},\bar{y}\in \conv (\Omega)} \!\max_{j,i=1,...,N}\! \max \{ \norm{ \partial_{j,1} \!\nabla_2  \partial_{i,2} \partial_{i,1} k(\bar{x},\bar{y})\!}_2\!,\!\norm{\!\partial_{j,2} \!\nabla_2 \partial_{i,2} \partial_{i,1} k(\bar{x},\bar{y})}_2 \}.
\end{align*}
Overall, the statement follows with $L =  \sqrt{2}( \sqrt{C_1}+ \sqrt{NC_2} ) \norm{v}_{k}$.
\end{proof}
\noindent Now we can   give a convergence result related to the VF and its derivative: 
\begin{theorem}\label{theorem:ConvVF}
Let Assumption \ref{ass:VfinRKHS} with $\Omega :=\mathcal{T}(\mathcal{A})$  and additionally $f,g \in C^1(\mathcal{T}(\mathcal{A}))$ hold, then for every $ T> 0 $ there is a constant $C_T >0$ independent of $\epsilon_{tol,d}$, $\epsilon_{tol,f}$ and the respective Hermite kernel interpolant $s_v^n$ for $ v $  such that
\begin{align*}
\max_{(t,x_0)  \in [0,T] \times \mathcal{A}} \big( &\vert v(\mathbf{x}^*(t;x_0)) - s_v^n(\mathbf{x}^*(t;x_0))\vert \\ +  &\norm{\nabla v(\mathbf{x}^*(t;x_0))-s_v^n(\mathbf{x}^*(t;x_0))}_{2} \big) \leq \epsilon_{tol,d} \,C_T   + \epsilon_{tol,f}
\end{align*}
if   Algorithm \ref{algo:greedy1} with $\epsilon_{tol,d}$ and  Algorithm \ref{algo:greedy2} with $\epsilon_{tol,f}$ have terminated. 
\end{theorem}
\begin{proof}
First, we fix an $\mathbf{x}^*(t_1;x_{0,1}) \in \mathcal{T}(\mathcal{A})$ with $t_1 \leq T$ and $x_{0,1} \in  \mathcal{A}$. Due to the  selection criterion in Algorithm \ref{algo:greedy1}, we know that for the state $ x_ {0,1} $ in $\mathcal{A}$ there is a trajectory $\mathbf{x}^*(\, \cdot\,;x_{0,2}) \subset \mathcal{D}$ in the available data set and a time $t_2 \geq 0$ with
\begin{align}\label{eq:ProofConvVF1}
\norm{x_{0,1} -\mathbf{x}^*(t_2;x_{0,2}) }_{2} \leq \epsilon_{tol,d} .
\end{align}
Next, we define $h(x) := f(x) -\frac{1}{2}g(x)R^{-1}g(x)^{\top} \nabla v(x)$, which is Lipschitz continuous with constant $L_h$ on the compact set $\mathcal{T}(\mathcal{A})$ as $f,g$ and $\nabla v$ are differentiable. Thus,  the ODE of the optimal trajectory in integral form for any $x_0 \in \mathcal{A}$ becomes
\begin{align*}
\mathbf{x}^*(t;x_{0}) = x_0 + \int_{0}^t h(\mathbf{x}^*(s;x_{0}) ) \text{d}s.
\end{align*}
Using this we get
\begin{align*}
&\norm{\mathbf{x}^*(t_1;x_{0,1})-\mathbf{x}^*(t_1+t_2;x_{0,2})}_{2} \\ \leq  & \norm{x_{0,1}-\mathbf{x}^*(t_2;x_{0,2})}_{2} + \int_0^{t_1} \norm{ h(\mathbf{x}^*(s;x_{0,1}))-h(\mathbf{x}^*(s+t_2;x_{0,2}))}_{2} \,\text{d}s \\
\leq &\norm{x_{0,1}-\mathbf{x}^*(t_2;x_{0,2})}_{2} + L_h \int_0^{t_1} \norm{ \mathbf{x}^*(s;x_{0,1})-\mathbf{x}^*(s+t_2;x_{0,2})}_{2} \,\text{d}s,
\end{align*}
which leads with Grönwall's lemma, Equation \eqref{eq:ProofConvVF1} and $t_1\leq T$ to 
\begin{align}\label{eq:ProofConvVF2}
\norm{\mathbf{x}^*(t_1;x_{0,1})\!-\!\mathbf{x}^*(t_1+t_2;x_{0,2})}_{2}\! \leq\! \norm{x_{0,1}\!-\!\mathbf{x}^*(t_2;x_{0,2})}_{2} e^{L_h t_1} \!\leq \! \epsilon_{tol,d} e^{L_h T}.
\end{align}
Furthermore after the termination of Algorithm \ref{algo:greedy2}, we know that
\begin{align}\label{eq:ProofConvVF3}
&\vert v(\mathbf{x}^*(t_1+t_2;x_{0,2}))-s_v^n(\mathbf{x}^*(t_1+t_2;x_{0,2}))\vert \notag \\ + & \norm{\nabla v(\mathbf{x}^*(t_1+t_2;x_{0,2}))- \nabla s_v^n(\mathbf{x}^*(t_1+t_2;x_{0,2}))}_{2} \leq \epsilon_{tol,f}
\end{align}
as  $\mathbf{x}^*(\, \cdot\,;x_{0,2}) \subset \mathcal{D}$ is in the available data set. So, overall we get 
\begin{align*}
&\vert v(\mathbf{x}^*(t_1;x_{0,1}))\!-\!s_v^n(\mathbf{x}^*(t_1;x_{0,1}))\vert \!+\! \norm{\nabla v(\mathbf{x}^*(t_1;x_{0,1}))\!-\! \nabla s_v^n(\mathbf{x}^*(t_1;x_{0,1}))}_{2} \\
 = & \vert v(\mathbf{x}^*(t_1;x_{0,1}))\!-\!v(\mathbf{x}^*(t_1\!+\!t_2;x_{0,2}))\!+\!v(\mathbf{x}^*(t_1\!+\!t_2;x_{0,2}))\!-\!s_v^n(\mathbf{x}^*(t_1;x_{0,1}))\vert \\
 & +  \big\| \nabla v(\mathbf{x}^*(t_1;x_{0,1}))-\nabla v(\mathbf{x}^*(t_1+t_2;x_{0,2}))\\ &+\nabla v(\mathbf{x}^*(t_1+t_2;x_{0,2}))-\nabla s_v^n(\mathbf{x}^*(t_1;x_{0,1}))\big\|_2 \\
 \leq & (L_v + L_{\nabla v})e^{L_h T} \epsilon_{tol,d} \\
 & + \vert v(\mathbf{x}^*(t_1+t_2;x_{0,2}))-s_v^n(\mathbf{x}^*(t_1;x_{0,1}))\vert \\
& +  \norm{\nabla v(\mathbf{x}^*(t_1+t_2;x_{0,2}))-\nabla s_v^n(\mathbf{x}^*(t_1;x_{0,1}))}_{2} \\
 = & (L_v + L_{\nabla v})e^{L_h T} \epsilon_{tol,d} \\
 & + \vert v(\mathbf{x}^*(t_1+t_2;x_{0,2}))-s_v^n(\mathbf{x}^*(t_1+t_2;x_{0,2}))\\ &+s_v^n(\mathbf{x}^*(t_1+t_2;x_{0,2}))-s_v^n(\mathbf{x}^*(t_1;x_{0,1}))\vert \\
& +  \big\|\nabla v(\mathbf{x}^*(t_1+t_2;x_{0,2}))- \nabla  s_v^n(\mathbf{x}^*(t_1+t_2;x_{0,2})) \\ &+ \nabla  s_v^n(\mathbf{x}^*(t_1+t_2;x_{0,2}))-\nabla s_v^n(\mathbf{x}^*(t_1;x_{0,1}))\big\|_2 \\
 \leq & (L_v + L_{\nabla v})e^{L_h T} \epsilon_{tol,d} + \epsilon_{tol,f} \\
 & + \vert s_v^n(\mathbf{x}^*(t_1+t_2;x_{0,2}))-s_v^n(\mathbf{x}^*(t_1;x_{0,1})) \vert \\
& +  \norm{\nabla  s_v^n(\mathbf{x}^*(t_1+t_2;x_{0,2}))-\nabla s_v^n(\mathbf{x}^*(t_1;x_{0,1}))}_{2} \\
\leq & \underbrace{(L_v + L_{\nabla v} + L)e^{L_h T}}_{=:C_T} \epsilon_{tol,d}   + \epsilon_{tol,f},
\end{align*}
where we inserted zero terms for the first inequality and then used the triangle inequality, the Lipschitz continuity of $v$ (Lipschitz constant $L_v$) and $\nabla v$ (Lipschitz constant $L_{\nabla v}$ and Equation \eqref{eq:ProofConvVF2} to obtain the first inequality. For the second equality we inserted other zero terms and then utilize the triangle inequality, Lemma \ref{lemma:surrogateLipschitz} and Equation \eqref{eq:ProofConvVF2} to get the second inequality. The last inequality results from \eqref{eq:ProofConvVF3}.  Since $\mathbf{x}^*(t_1;x_{0,1})$ was chosen arbitrarily, the above inequality also holds for the maximum, and so the statement follows.
\end{proof}
\noindent Note that the latter theorem is valid for any fixed finite time horizon.  \\
\newline
\noindent The next result concerns the convergence of the surrogate controlled trajectory $\mathbf{x}_s$ to the optimal trajectory $\mathbf{x}^*$. In order to establish the convergence, it is required that the regularity of the VF applies to a slightly larger domain. Therefore, for a fixed $\delta>0$ we define the dilated set $\mathcal{T}_{\delta}(\mathcal{A}):= \mathcal{T}(\mathcal{A}) + \overline{B_{\delta}}$, where $\overline{B_{\delta}}$ is the closed ball centered at zero with radius $\delta$. 
\begin{theorem}\label{theorem:ConvTraj2}
Let Assumption \ref{ass:VfinRKHS} with $\Omega:= \mathcal{T}_{\delta}(\mathcal{A})$ hold for a $\delta>0$ and additionally $f,g \in C^1(\mathcal{T}_{\delta}(\mathcal{A}))$, then for every $ T> 0 $ there are two constants $\overline{C}_T >0$ and $\tilde{C}_T >0$ independent of $\epsilon_{tol,d}$, $\epsilon_{tol,f}$ and the respective Hermite kernel interpolant $s_v^n$ for $ v $  such that
\begin{align*}
 \max_{(t,x_0)  \in [0,T] \times \mathcal{A}} \norm{\mathbf{x}^*(t;x_0)-\mathbf{x}_s(t;x_0)}_{2} \leq \epsilon_{tol,d} \, \overline{C}_T  + \epsilon_{tol,f} \, \tilde{C}_T,
\end{align*}
if   Algorithm \ref{algo:greedy1} with $\epsilon_{tol,d}$ and  Algorithm \ref{algo:greedy2}  with $\epsilon_{tol,f}$ have terminated. Furthermore, it is assumed that the thresholds $\epsilon_{tol,d}, \epsilon_{tol,f}$ are so small that  $\epsilon_{tol,d} \, \overline{C}_T  + \epsilon_{tol,f} \, \tilde{C}_T < \delta$ holds.
\end{theorem}
\begin{proof}
We start by defining a function that summarizes the dynamics of the ODE systems under consideration for a given VF surrogate $e(x)$:
\begin{align*}
h_e(x):=  f(x) -\frac{1}{2}g(x)R^{-1}g(x)^{\top} \nabla e(x)
\end{align*}
Note that the function $h_{s_v^n}$ is Lipschitz continuous on $\mathcal{T}_{\delta}(\mathcal{A})$ where the Lipschitz constant $L_{h_s}$ is independent of the respective interpolant of the VF. This follows with the differentiability of $f,g$ and $\nabla v$. The independence of the constant from the surrogate follows as in Lemma \ref{lemma:surrogateLipschitz}. Due to Kirszbraun's theorem \cite{kirszbraun1934} there is a globally Lipschitz continuous function $\overline{h}_{s_v^n}: \mathbb{R}^N \longrightarrow \mathbb{R}^N$ with the same Lipschitz constant $L_{h_s}$ and $\overline{h}_{s_v^n}(x) = h_{s_v^n}(x)$ for $x\in \mathcal{T}_{\delta}(\mathcal{A})$. Next, an auxiliary trajectory $\overline{\mathbf{x}}_s(t;x_0)$ is defined which satisfies the ODE
\begin{align*}
\dot{\overline{\mathbf{x}}}_s(t;x_0) = \overline{h}_{s_v^n}(\overline{\mathbf{x}}_s(t;x_0))
\end{align*}
with $\overline{\mathbf{x}}_s(0;x_0) = x_0$. With the Picard-Lindelöf theorem, such a trajectory exists for all time intervals, due to the global Lipschitz continuity of $\overline{h}_{s_v^n}$. In a further step, we show that the auxiliary trajectory and the optimal  trajectory  are getting closer to each other as $\epsilon_{tol,d}$ and $\epsilon_{tol,f}$ decrease. Let us also use the definition $G(x) = -\frac{1}{2}g(x)R^{-1}g(x)^{\top}$ here to simplify. For a fixed $(t;x_0)$ with $t\leq T$, it holds
\begin{align*}
     & \norm{\mathbf{x}^*(t;x_0)\!-\!\overline{\mathbf{x}}_s(t;x_0)}_2 \\ 
\leq & \int_{0}^t\!\norm{ h_v(\mathbf{x}^*(t;x_0) )\! - \!\overline{h}_{s_v^n}(\overline{\mathbf{x}}_s(t;x_0))}_2\! \text{d}t \\
= & \int_{0}^t \! \norm{ h_v(\mathbf{x}^*(t;x_0) )\! -\!\overline{h}_{s_v^n}(\mathbf{x}^*(t;x_0) )\!+\!\overline{h}_{s_v^n}(\mathbf{x}^*(t;x_0) )\!-\! \overline{h}_{s_v^n}(\overline{\mathbf{x}}_s(t;x_0))}_2\! \text{d}t \\
\leq & \int_{0}^t \! \norm{ h_v(\mathbf{x}^*(t;x_0) )\! -\!\overline{h}_{s_v^n}(\mathbf{x}^*(t;x_0) )}_2\! \text{d}t \!+\! \int_{0}^t\! \norm{\overline{h}_{s_v^n}(\mathbf{x}^*(t;x_0) )\!- \!\overline{h}_{s_v^n}(\overline{\mathbf{x}}_s(t;x_0))}_2 \!\text{d}t \\
\leq & \int_{0}^t \! \underbrace{\norm{ h_v(\mathbf{x}^*(t;x_0) )\! -\!\overline{h}_{s_v^n}(\mathbf{x}^*(t;x_0) )}_2}_{=\norm{G(\mathbf{x}^*(t;x_0))( \nabla v(\mathbf{x}^*(t;x_0))\! -\! \nabla s_v^n(\mathbf{x}^*(t;x_0))}_2}\!\!\!\!\!\! \text{d}t\! +\! L_{h_s}\!\int_{0}^t \norm{\mathbf{x}^*(t;x_0) \!-\! \overline{\mathbf{x}}_s(t;x_0)}_2\! \text{d}t \\
\leq & (\epsilon_{tol,d} \,C_T   + \epsilon_{tol,f}) \max_{x\in \mathcal{T}(\mathcal{A})} \norm{G(x)}_2 T + L_{h_s}\!\int_{0}^t \!\norm{\mathbf{x}^*(t;x_0)\! -\! \overline{\mathbf{x}}_s(t;x_0)}_2 \!\text{d}t ,
\end{align*}
where we utilized Theorem \ref{theorem:ConvVF} for the last inequality. So with the Grönwall's inequality \cite{gronwall1919}, the definitions $ \overline{C}_T := C_T \max_{x\in \mathcal{T}(\mathcal{A})} \norm{G(x)}_2 T e^{L_{h_s} T}$ and $\tilde{C}_T := \max_{x\in \mathcal{T}(\mathcal{A})} \norm{G(x)}_2 T e^{L_{h_s} T}$,  it follows
\begin{align*}
\norm{\mathbf{x}^*(t;x_0)-\overline{\mathbf{x}}_s(t;x_0)}_2 \leq \epsilon_{tol,d} \, \overline{C}_T  + \epsilon_{tol,f} \, \tilde{C}_T \leq \delta. 
\end{align*}
Here the latter inequality comes from the assumption. The crucial point now is that because of the estimate 
\begin{align*}
  \norm{\mathbf{x}^*(t;x_0)-\overline{\mathbf{x}}_s(t;x_0)}_2  \leq \delta  
\end{align*}
it holds that $\overline{\mathbf{x}}_s(t;x_0) \in \mathcal{T}_{\delta}(\mathcal{A})$ for all $t\in[0,T]$ and therefore $\overline{\mathbf{x}}_s(t;x_0) =  \mathbf{x}_s(t;x_0)$ as $\overline{h}_{s_v^n}(x) = h_{s_v^n}(x)$ for $x\in \mathcal{T}_{\delta}(\mathcal{A})$. Thus, the statement follows as $(t;x_0)$ was chosen arbitrarily.
\end{proof}
\noindent From a practical point of view, the last theorem is very interesting because it gives a guarantee that the approximate optimal feedback control will provide near optimal trajectories if the Algorithms \ref{algo:greedy1} and \ref{algo:greedy2} have terminated for sufficiently small thresholds $\epsilon_{tol,d} $ and $\epsilon_{tol,f}$.
\section{Numerical experiments}
\noindent Three model problems are considered in this section. For each of them, a training data set is computed using  Algorithm \ref{algo:greedy1}. Because the optimal open-loop solution becomes slightly less accurate near zero, as indicated by the first model problem, all trajectories are truncated after time $T$. Algorithm \ref{algo:greedy2} generates the Hermite surrogate and the structured Hermite surrogate, adapting step 4 for the latter. For the former surrogate, we use the Wendland kernel $k(x,y) = \Phi_{N,m}(\gamma \norm{x-y}_2)$ with
\begin{align*}
\Phi_{N,m} = (1-r)_+^{l+2} \left[(l^2 + 4l + 3)r^2 + (3l + 6)r + 3\right],
\end{align*}
$l := \lfloor N \slash 2 \rfloor+3$ and $(1-r)_+ := \max\{1-r,0\}$ from  Corollary 9.14 in \cite{wendland2004}, which exactly satisfies the regularity assumptions in Assumption \ref{ass:VfinRKHS}. Here $1 \slash \gamma >0$ is the kernel width. We multiply this Wendland kernel by $\langle x,y \rangle^2$ as mentioned in  Section \ref{sec:strucHermite} to obtain the kernel for the structured Hermite surrogate. The VF surrogates are assessed with respect to a test error for the approximation of the VF and the quality of the resulting surrogate controlled trajectories. The required test set consists of 20 open-loop trajectories for which the initial states $\{x_{0,1},...,x_{0,20} \}$ are selected in $\mathcal{A}$ using a so-called geometric greedy procedure \cite{demarchi2005}.  To determine the quality of the feedback realized by the surrogate, we use a mean relative $L^2$-error  
\begin{align*}
MRL^2Error(\mathbf{x}_{\text{OL}},\mathbf{x}_{\text{SR}}) \!:=\! \frac{1}{20}\! \sum_{j=1}^{20} \!\sqrt{\frac{\sum_{i=0}^{N_j} \! \norm{\mathbf{x}_{\text{OL}}(t_{j,i}; x_{0,j})\!-\!\mathbf{x}_{\text{SR}}(t_{j,i}; x_{0,j})\!}^2_2 }{\sum_{i=0}^{N_j} \! \norm{\mathbf{x}_{\text{OL}}(t_{j,i}; x_{0,j})\!}^2_2}}
\end{align*}
for  partitionings $0 = t_{j,0} < t_{j,1} < ... t_{j,N_j} = T$, which are determined adaptively by the solver \emph{solve{\_}bvp} from SciPy, for $j=1,...,20$. It compares the optimal open-loop controlled trajectories $\mathbf{x}_{\text{OL}}$ of the test set with the surrogate controlled trajectories $\mathbf{x}_{\text{SR}}$. The latter is computed by solving the approximate optimal feedback control in closed-loop form
\begin{align*}
 \dot{\mathbf{x}}_{\text{SR}}(s;\!\! x_{0,j})\! = & f(  \mathbf{x}_{\text{SR}}(s; \!\!x_{0,j})) \\ &-\!\frac{1}{2} g(  \mathbf{x}_{\text{SR}}(s; \!\!x_{0,j}))R^{-1}\!g(\mathbf{x}_{\text{SR}}(s;\!\! x_{0,j}))^{\top} \nabla s_v(\mathbf{x}_{\text{SR}}(s; \!\!x_{0,j}))  
\end{align*}
with $\mathbf{x}(0) = x_{0,j}$ for all $j=1,...,20$ using SciPy's \textit{solve{\_}ivp}. This error is also used to determine the hyperparameter $\gamma$, which is realized with a  $s$-fold cross-validation, where the training set of the trajectories is divided into $s$ sets of complete trajectories. To allow a better comparison of the model problems with each other, we introduce the quantities $C_{\text{max},\mathcal{A}}:= \max_{x\in \mathcal{A}} \norm{x}_{2} $ and $C_{\text{max},v}:= \max_{x\in \mathcal{D}}\left(\norm{ v(x)}_{2}+  \norm{\nabla v(x)}_{2} \right) $ to specify a relative cover distance and training error in the following. All runtimes of the numerical experiments refer to a laptop with an AMD Ryzen 9 5900HX CPU and 16GB RAM.
\subsection*{An academic model problem}
The first model problem is an academic model problem (AMP), which is not application motivated but dimension-variable and has the advantage that the VF is known analytically, allowing exact error assessment. It has the form
\begin{gather}
\min_{\mathbf{u} \in \mathcal{U}_{\infty}} J(\mathbf{u}) = \min_{\mathbf{u} \in \mathcal{U}_{\infty}} \int_0^{\infty}  \alpha e^{\norm{\mathbf{x}(s)}_2^2} \langle \mathbf{x}(s),\mathbf{x}(s) \rangle_2^2 + \beta \left(\mathbf{u}(s)\right)^2 \,\text{d}s \label{eq:aca1} \\ 
\text{ subject to } \dot{\mathbf{x}}(s) \!= \! \norm{\mathbf{x}(s)}_2^2 \mathbf{x}(s) + e^{- \frac{\norm{\mathbf{x}(s)}_2^2}{2}} \mathbf{x}(s) \mathbf{u}(s)  \text{ and }\mathbf{x}(0) \!= \!x_0 \in \mathbb{R}^N  \label{eq:aca3}
\end{gather}
with controller dimension $M=1$. This model problem may also be useful for other studies. The HJB equation for this problem is
\begin{align}\label{eq:HJBAMC}
\langle \nabla v(x),x \rangle_2 \norm{x}_2^2  - \frac{e^{- \norm{x}_2^2}}{4 \beta} \langle \nabla v(x),x \rangle_2^2 +\alpha e^{\norm{x}_2^2} \norm{x}_2^4 = 0.
\end{align}
For the VF we chose the ansatz $v(x) = C e^{\norm{x}^2_2} -C$  and thus $ \nabla v(x) = 2 C e^{\norm{x}^2_2} x$. Inserting this into \eqref{eq:HJBAMC} gives
\begin{align*}
\left(2C - C^2 \frac{1}{ \beta}  +\alpha \right)  e^{\norm{x}^2_2}  \langle  x,x \rangle_2^2  = 0,
\end{align*}
which is solved for $C_{1 \backslash 2 } := \beta (1 \pm \sqrt{1+\alpha \slash \beta } )$. But only $C :=  \beta (1 + \sqrt{1+\alpha \slash \beta} )$ makes sense, because the VF must always be positive.  The optimal closed-loop system in terms of  \eqref{eq:aca1} is therefore
\begin{align*}
\dot{\mathbf{x}}^*(s) &= \left(\norm{\mathbf{x}^*(s)}_2^2  - \frac{1}{2 \beta} e^{- \norm{\mathbf{x}^*(s)}_2^2}  \langle \mathbf{x}^*(s), \nabla v(x) \rangle_2 \right) \mathbf{x}^*(s) \\
& = -\sqrt{1+\frac{\alpha}{\beta}} \norm{\mathbf{x}^*(s)}_2^2 \mathbf{x}^*(s),
\end{align*}
which is globally asymptotically stable, since a Lyapunov function is given by  $v$ (see \cite[Chapter 4]{Khalil2013} for more background information on Lyapunov stability). The standard linearization approach to get an approximation of the VF in a neighborhood of  zero, as required for the structured Hermite surrogate, is not possible for this problem, since $J_f(0) = 0$ and $g(0) = 0 $. Nevertheless, a quadratic approximation can also be obtained by truncating the Taylor series of the true VF:
\begin{align*}
v(x) \approx v(0) + \nabla v(0) (x-0) + (x-0)^{\top}J_{\nabla v}(0)(x-0) =   2C \norm{x}^2_2
\end{align*} 
The approximate optimum closed-loop system is therefore
\begin{align*}
\dot{\mathbf{x}}^*(s)  = \left(1-\frac{2C}{\beta} e^{-\norm{\mathbf{x}^*(s)}_2^2} \right)  \norm{\mathbf{x}^*(s)}_2^2 \mathbf{x}^*(s), 
\end{align*}
which is only guaranteed to be asymptotically stable if
\begin{align*}
\left(1-\frac{2C}{\beta} e^{-\norm{\mathbf{x}^*(s)}_2^2} \right) <0,
\end{align*}
which leads to 
\begin{align*}
 \norm{\mathbf{x}^*(s)}_2 < \sqrt{ \log\left(2+2\sqrt{1+\frac{\alpha}{\beta}}\right)}:=r_1,
\end{align*}
since  the VF $v$ is again a Lyapunov function on the Ball $B_{r_1}(0)$. \\
\newline
\noindent Note that the construction of the model problem works for arbitrary state dimension $N$. For the following experiments, we set $N=2$ and $\mathcal{A}:= [-1,1]^N$ for the set of initial states. Moreover, the design parameters of the problem are chosen to be $\alpha = 10^5$ and $\beta = 1$, such that the high ratio $\frac{\alpha}{\beta}$ ensures that the system quickly approaches zero. Fig.\ \ref{fig:AMPDataSet} shows the training (turquoise) and test (orange) data sets for the academic model problem, using 100 trajectories for the former, which then has a cover distance smaller than $\epsilon_{\text{tol,d}} \slash C_{\text{max},\mathcal{A}} = 1.3 \cdot 10^{-3}$ in Algorithm \ref{algo:greedy1}. The relative deviation of the  computed data to the values of the true VF and its gradient is smaller than $10^{-7}$. Without truncating the trajectories to $T=99$, it would be $10^{-5}$. Since only here the true VF is accessible, we keep this value for $T$ also for the other model problems. Another indicator of data quality is the fulfillment of the HJB equation along the trajectories. Here the largest deviation is smaller than $10^{-12}$. Since this is close to machine precision, the Hermite kernel interpolation can then also be interpreted as a collocation method of the HJB equation. The average time taken to compute an optimal open-loop trajectory is approximately 42 seconds. For cross validation, $s=10$ was used  to determine the kernel shape parameter. This yields $\gamma_{H}= 0.04$ for the Hermite surrogate. With a total number of 200, the centers are scattered like the red dots in Fig.\ \ref{fig:AMPDataSet} (left). Using Algorithm \ref{algo:greedy2}, the selection took 940 seconds and resulted in a relative  training error smaller than $\epsilon_{\text{tol,f}}\slash C_{\text{max},v} =  5.089 \cdot 10^{-4}$. For the structured Hermite surrogate, the cross validation leads to $\gamma_{SH}= 0.4$. Here the 200 centers are selected as the blue dots in Fig.\ \ref{fig:AMPDataSet} (left), which took 893 seconds. In this, the relative training error is smaller than $\epsilon_{\text{tol,f}}\slash C_{\text{max},v} =  6.798 \cdot 10^{-4}$.
\begin{figure}[ht]%
\centering
\includegraphics[width=0.9\textwidth]{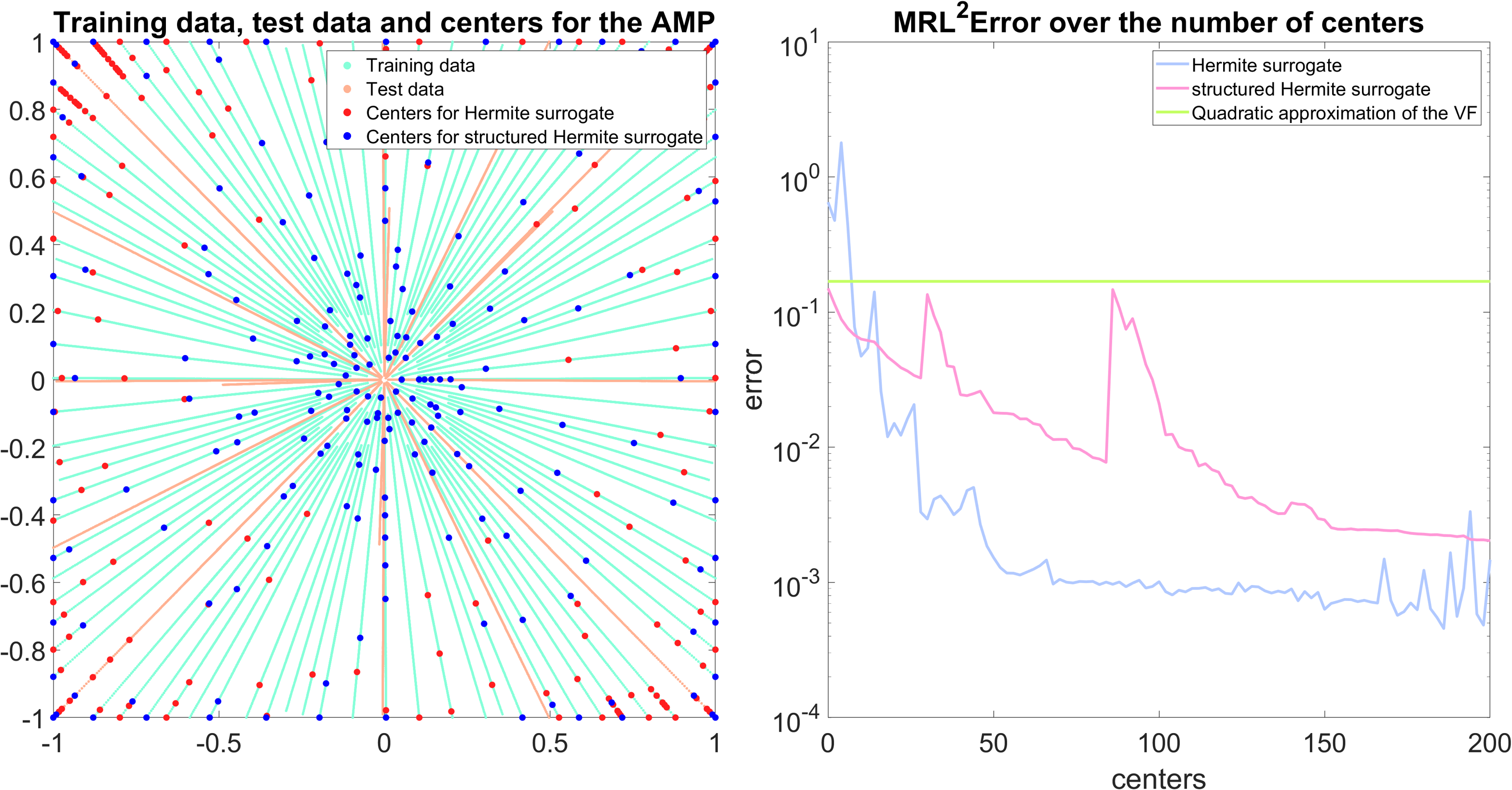}
\caption{The left diagram shows the training data, the test data and the centers chosen by Algorithm \ref{algo:greedy2} for the Hermite and the structured Hermite surrogate. On the right diagram we present a semilogy plot of the $MRL^2Error$ for the three types of surrogate over the number of centers selected.} \label{fig:AMPDataSet}
\end{figure}
In terms of the $MRL^2Error$ (see Fig.\ \ref{fig:AMPDataSet} right), it can be seen that both data-based Hermite surrogates produce smaller errors than the quadratic approximation, whose $MRL^2Error$ is around $1.688 \cdot 10^{-1} $. For a few centers, the structured Hermite surrogate is the best, as it gives a $MRL^2 error$ of $6.757 \cdot 10^{-2} $ after $8$ selected centers. However, it is quickly caught up by the Hermite surrogate, which then overall results in a lower  $MRL^2Error$. Its minimum is at $186$ centers with a value of $4.556 \cdot 10^{-4}$. The minimal $MRL^2Error$ for the structured Hermite surrogate is at  $200$ centers with a value of $2.033 \cdot 10^{-3}$. It is also clearly visible in Fig.\ \ref{fig:AMPDataSet} (right) that for both surrogates the error is beginning to stagnate, which is because there are trajectories in the test data that are not seen in the training (see Fig.\ \ref{fig:AMPDataSet} left). Therefore, the error could only be reduced by having data with a smaller $ \epsilon_{\text{tol,d}}$, which fits to the convergence results in   Theorem \ref{theorem:ConvTraj2}. The calculation of a surrogate control trajectory for an initial state in the test data set took on average $0.96$ seconds for the Hermite surrogate and $1.16 $  seconds for the structured Hermite surrogate. So here we have a considerable advantage compared to the runtime of $42$ seconds mentioned above for an optimal open-loop solution, while achieving almost the same result. 
\subsection*{Gripper-Soft-Tissue}
The second model problem is the Gripper Soft Tissue (GST). We only give a very concise description of the model, for more details we refer to \cite{ehring2021}. The GST is a two-dimensional physical domain model that describes a gripper that has gripped a soft tissue, such as a fruit or a piece of meat, and brings it to a prescribed target position while avoiding an obstacle.  Figure \ref{fig:fig1} shows a schematic representation. The governing equation of a linear elastic body is used to model the displacement field of the soft tissue. The gripper is modelled by a point mass whose displacement is described by Newton's second law.  An external force $\mathbf{u}(s)$ controls the gripper in this process. Consequently, the controller dimension is $M=2$. A semi-discretized version of the coupled system can be represented by an ODE system utilizing finite element methods. For this problem we choose $ N = 36 $.  The target position is the state with no displacement and no velocity. The state constraints for the obstacle placed in $(4,0)$ with a radius of $1.8$ are handled by an external penalty method. This means that, an additional term is introduced into the running payoff that makes states inside the obstacle very expensive and thus the optimal trajectory bypasses the obstacle. The OCP reads
\begin{gather*}\label{eq:MPfinal1} 
 \min_{\mathbf{u} \in \mathcal{U}_{\infty}} \int_0^{\infty} \mathbf{x}(s)^{\top}Q \mathbf{x}(s)  + \mathbf{x}(s)^{\top} \mathbf{x}(s) \sum_{i=1}^c e^{3.5 C_i(\mathbf{x}(s))}+ \mathbf{u}(s)^{\top} R \mathbf{u}(s) \,\text{d}s  \\ 
\text{ subject to } E \dot{\mathbf{x}}(s) = A\mathbf{x}(s)  + D \mathbf{u}(s) \text{ and }\mathbf{x}(0) = x_0, \label{eq:SCfinal1}
\end{gather*}
where the functions $(C_i)_{i=1}^c$  are positive if the $ i $th-node is in the obstacle and negative otherwise. Note that the problem without obstacle is a standard linear–quadratic regulator (LQR) whose solution can be computed by solving the algebraic Riccati equation. This was used to compute the quadratic approximation of the VF. 
\begin{figure}[ht]%
\centering
\includegraphics[width=0.6\textwidth]{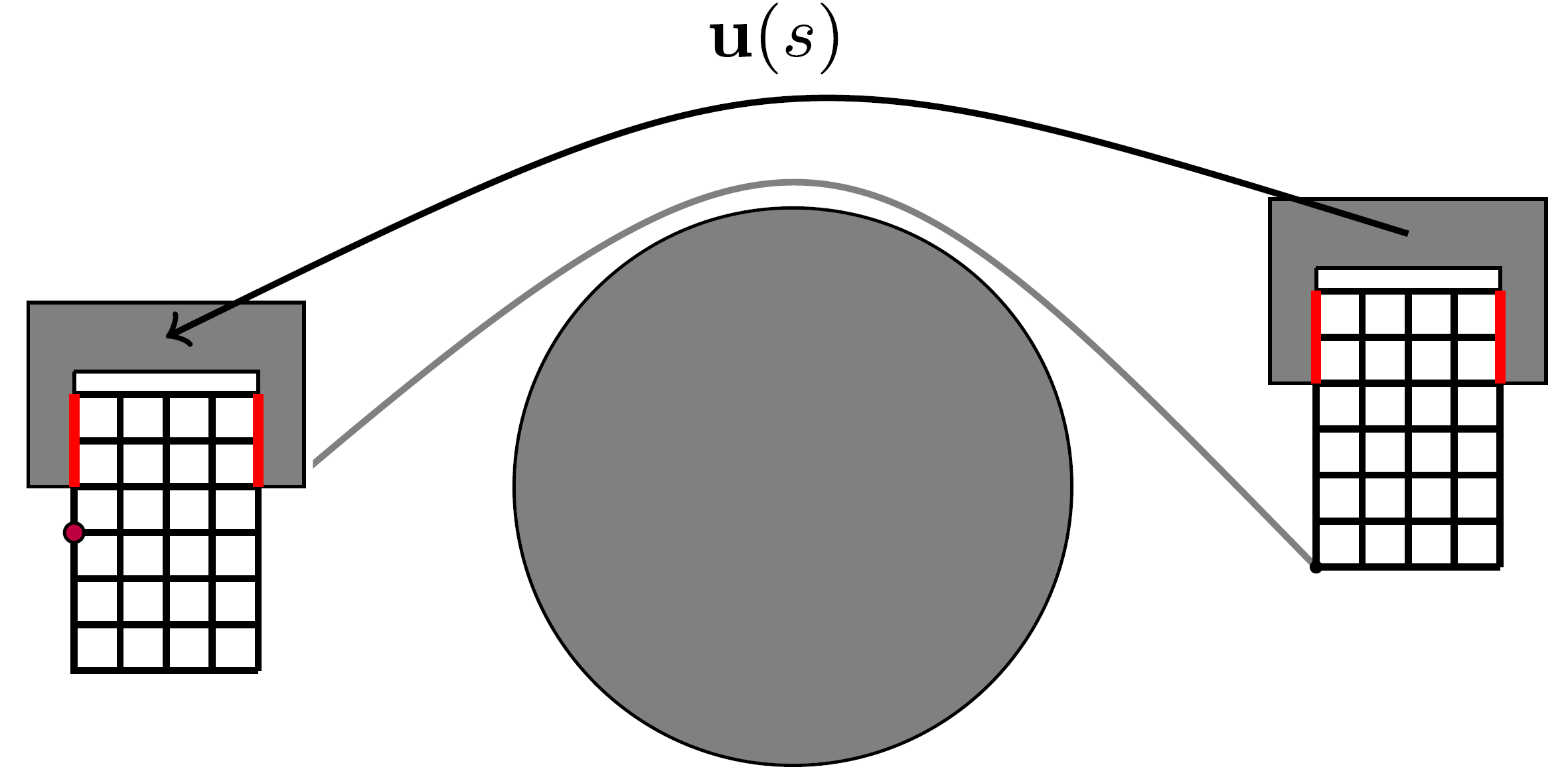}
\caption{A gripper that transports soft tissue to a target point while avoiding an obstacle. The purple dot indicates the node which is zero when the system is in a position with no displacement.}\label{fig:fig1}
\end{figure}
For reasons of dimensionality reduction, the set of initial states $\mathcal{A}$ is chosen to contain only states with constant displacement and velocity.  So there is no initial state with a deformed soft tissue. This results in four degrees of freedom for displacement and velocity in the $x$ and $y$ directions, which are in $[-2,10] \times [0 ,6] \times [-2,2] \times [-2,2]$ and where the states touching or inside the obstacle are removed. Note that only initial states that have their optimal path above the obstacle are included in this set. If we would allow trajectories below the obstactle, there would be uniqueness problems when computing the optimal solution with the necessary PMP conditions. Fig.\ \ref{fig:GSTDataSet} illustrates the training (turquoise) and test (orange) data set, showing only the position of the purple node from Fig.\ \ref{fig:fig1}. The training data set consists  of 250 trajectories, which have a cover distance smaller than $\epsilon_{\text{tol,d}}\slash C_{\text{max},\mathcal{A}} =1.2 \cdot 10^{-2}$ in  Algorithm \ref{algo:greedy1}. The maximum deviation from the HJB equation on the trajectories in the data set is smaller than $10^{-7}$ and the calculation of an optimal trajectory takes about $110$ seconds. For cross validation, $s=10$  results in $\gamma_{H}= 0.003$ for the Hermite surrogate and $\gamma_{SH}= 0.015$ for the structured Hermite surrogate. For both, 300 centers were selected, distributed as illustrated by the blue dots in Fig.\ \ref{fig:GSTDataSet} (left) for the Hermite surrogate and by the red dots for the structured Hermite surrogate. The training of the first surrogate took 1546 seconds resulting in a relative training error less than $\epsilon_{\text{tol,f}}\slash C_{\text{max},v}=1.466 \cdot 10^{-2}$. The second surrogate was trained in 1405 seconds and yields a relative training error less than $\epsilon_{\text{tol,f}} \slash C_{\text{max},v}=5.155\cdot10^{-3}$.
\begin{figure}[ht]
\centering
\includegraphics[width=0.9\textwidth]{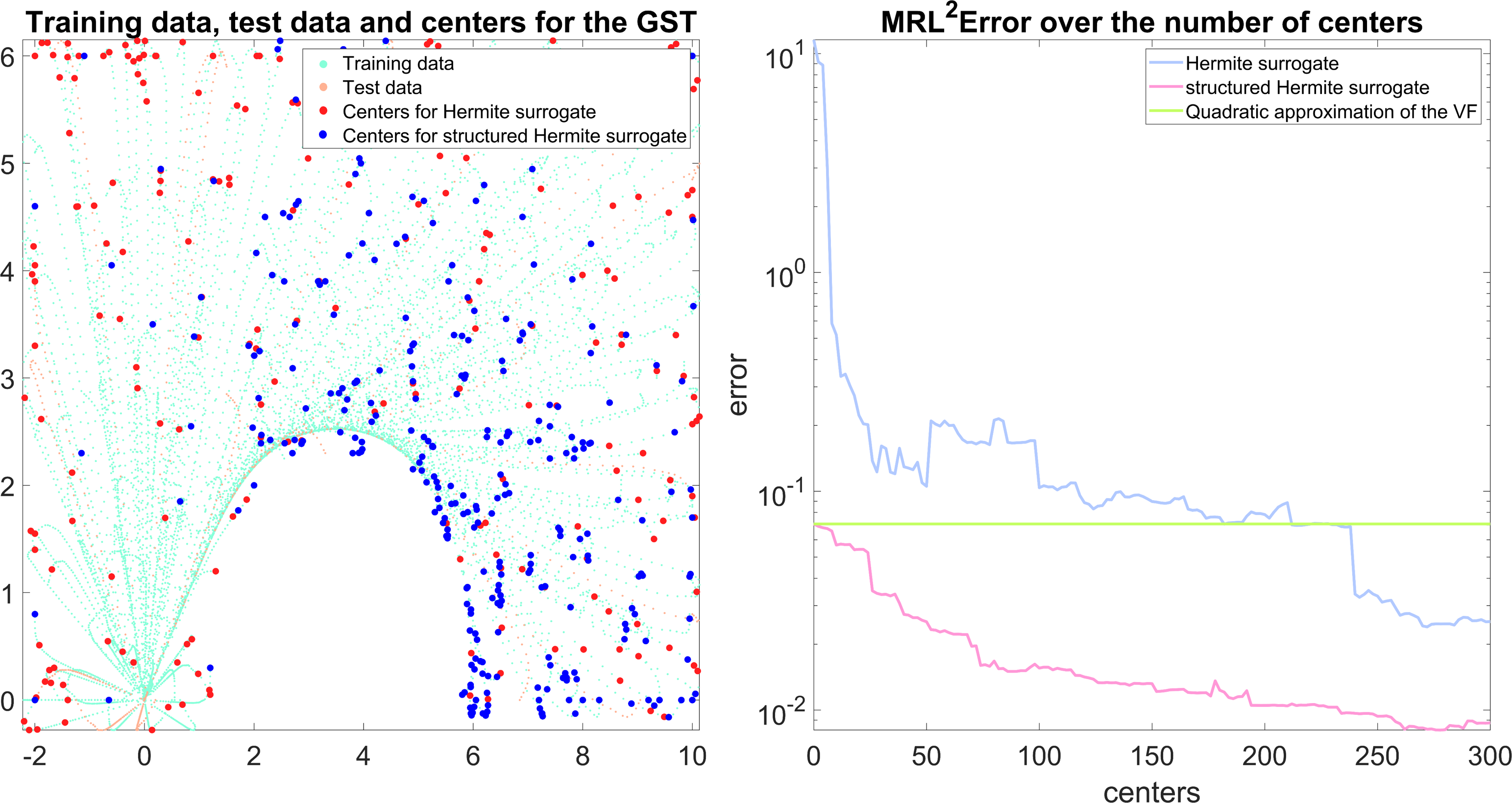}   
\caption{The left diagram shows the training data, the test data and the centers chosen by  Algorithm \ref{algo:greedy2} for the Hermite and the structured Hermite surrogate. The right diagram is a semilogy plot of the $MRL^2Error$ for the three types of surrogate over the number of centers selected.} 
\label{fig:GSTDataSet}
\end{figure}
\noindent  Looking at the $MRL^2Error$, we see that both surrogates can beat the control of the quadratic approximation, as the latter leads to an error of $7.088 \cdot 10^{-2} $. As this quadratic VF is the solution of the OCP without penalty term, the trajectories controlled by this surrogate do not avoid the obstacle. What is striking about this model problem is that the structured Hermite surrogate with a minimum $MRL^2Error$ of $8.099 \cdot 10^{-3}$ clearly outperforms the Hermite surrogate, which at its best has $2.399 \cdot 10^{-2} $. By design, the former starts with a much better error. In addition, very few centers are selected near the zero state as expected. Again, the error begins to stagnate for both surrogates, which can only be improved by reducing $\epsilon_{\text{tol,d}}$. Here, the computation of a surrogate controlled trajectory with an initial state in the test data set took on average $7.6 $ seconds for the Hermite surrogate and $2.8 $ seconds for the structured Hermite surrogate. Thus, also for this model problem, computing a trajectory based on the surrogate is much faster than computing an optimal open-loop solution with the PMP conditions, which took about $110$ seconds as mentioned above.
\subsection*{Nonlinear Heat equation}
The third model problem  we consider is a nonlinear heat equation (NHE) of Zeldovich type \cite[p.~2]{gilding2004}, which we adapted from \cite{Alla2023}. It is of the form 
\begin{align*}
\dot{y}(\xi,t) = \alpha \Delta y(\xi,t)+\beta(y^2(\xi,t)-y^3(\xi,t))+\Xi(\xi) u(\xi,t)
\end{align*}
for $(\xi,t) \in \Omega \times [0,\infty)$ with the boundary conditions
\begin{align*}
\partial_n y(\xi,t) = 0  \text{ for }(\xi,t) \in \partial \Omega \times  [0,\infty) \text{ and }
y(\xi,0) =  y_0(\xi)  \text{ for }\xi \in  \Omega, 
\end{align*}
$\Omega = (0,1) \times (0,1)$, $\alpha=5$ and $\beta=0.5$. Here $\Xi$ is the indicator function of the set $[0.25,0.75] \times [0.25,0.75]$.   The finite difference method can be used to semi-discretise the NHE. This leads to an ODE system involving the matrix $ A \in \mathbb{R}^{N \times N} $, which is the discretized version of the operator $ \alpha \Delta $ and the matrix $ B \in \mathbb{R}^{N \times M} $, that is the discretized version of $\Xi$. The quadratic and cubic terms become a component-wise application of this operation on the state. For the experiments we choose $ N = 100 $. This gives  $ M = 36 $.  The system has a stable equilibrium point in the constant one function and an unstable equilibrium point in the constant zero function. The goal of the control is to optimally steer the system to zero.  We fix the following OCP:
\begin{gather*}
\min_{\mathbf{u} \in \mathcal{U}_{\infty}} \int_0^{\infty} \norm{\mathbf{x}(s)}_{2} ^2 + 10^{-3}\norm{\mathbf{u}(s)}_{2}^2 \,\text{d}s   \\ 
\text{ subject to } \dot{\mathbf{x}}(s) = A\mathbf{x}(s)+ \beta(\mathbf{x}^2(s)-\mathbf{x}^3(s))+ B \mathbf{u}(s)   \text{ and }\mathbf{x}(0) = x_0 . 
\end{gather*}
\noindent We choose the space-continuous version of the set of initial states as
\begin{align*}
\{ a \sin(b\pi\xi_1)^2\sin(c\pi\xi_2)^2 &+ d \sin(e\pi\xi_1^2)^2\sin(f\pi\xi_2)^2 \\ &\,\,\,\vert a,b \in [-0.25,0.5] \text{ and } b,c,e,f \in \{1,2\} \}, 
\end{align*}
since the functions satisfy the Neumann boundary condition. We have to make this restriction for dimensionality reasons.  Then the set $\mathcal{A}$ is given by the spatially discretized version of these functions according to the finite difference grid. A meaningful visualization of the training and test data sets, as well as the centers selected for each of the two Hermite surrogates, is difficult for this model problem.  For illustrative purposes, Fig.\ \ref{fig:NHEDataSet} (left) only shows the course of the first optimal trajectory in the training data set. The latter  consists of 250 trajectories, with a cover distance smaller than $\epsilon_{\text{tol,d}}\slash C_{\text{max},\mathcal{A}} = 3.956 \cdot 10^{-2}$. Computing an optimal trajectory with PMP conditions took about $344$ seconds. The maximum deviation of the optimal trajectory from the HJB equation is smaller than $10^{-10}$. The cross validation  with $s=10$ resulted in  $\gamma_{H}= 0.002$ for the Hermite surrogate and $\gamma_{SH}= 0.02$ for the structured Hermite surrogate. Again, we use 200 centers for both model problems. For the Hermite kernel surrogate, the training phase took $2192$ seconds for the parameter  $\gamma_{H}$ and the relative training error fell below $\epsilon_{\text{tol,f}}\slash C_{\text{max},v}=7.326 \cdot 10^{-4}$. For the structured Hermite surrogate, the relative training error is less than $\epsilon_{\text{tol,f}}\slash C_{\text{max},v}=2.119\cdot 10^{-3}$. This took $2141$ seconds to train. 
\begin{figure}[ht]
\centering
\includegraphics[width=0.9\textwidth]{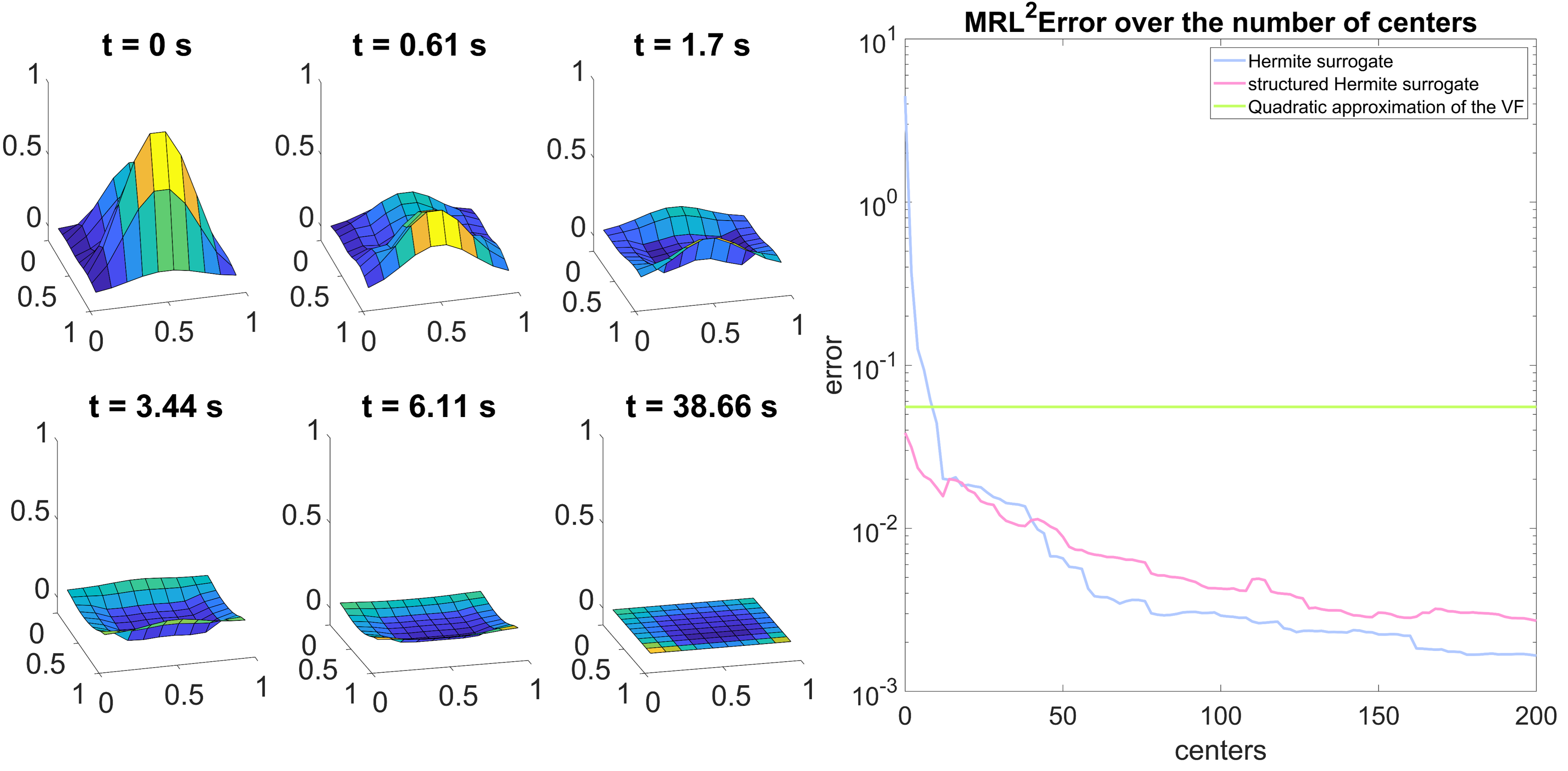} 
\caption{The left diagram shows the first trajectory in the training data set on six time instants. The right diagram presents a semilogy plot of the $MRL^2Error$ for the three types of surrogate over the number of centers selected.} 
\label{fig:NHEDataSet}
\end{figure}
In terms of the $MRL^2Error$, both surrogates quickly give better results than the quadratic approximation to the VF, which delivers $5.556 \cdot 10^{-2} $. In this model problem, the latter is constructed by linearizing the ODE and solving the resulting Algebraic Riccati equation. As with all model problems, the structured Hermite surrogate is better at the beginning, but similar to the first model problem, it is caught up by the Hermite surrogate. However, the structured Hermite surrogate gives a similar $MRL^2Error$ curve whose minimum value is $2.713 \cdot  10^{-3} $. The minimum $MRL^2Error$ value for the Hermite surrogate is $1.658 \cdot 10^{-3} $. It took on average $3.5 $ seconds to compute a surrogate controlled trajectory with an initial state in the test data set for the Hermite surrogate and $3.6 $  seconds for the structured Hermite surrogate, which again is much less than the $344$ seconds for computing an optimal open-loop controlled solution.  

\section{Conclusion}
A key element of our data-based approach is the generation of high quality VF data. This is achieved by a method that solves the infinite time horizon problem  by transforming the integration domain and adjusting the boundary conditions accordingly. An indicator that this method works very well is the small deviation of the data from the HJB equation of the infinite time horizon OCP. In fact, this was so small for some model problems that one could also interpret it as a collocation approach. Another aspect that makes Hermite interpolation possible in the first place is the matrix-free approach. Here, it should be emphasized that despite the $(N+1)n$ degrees of freedom of the system, the matrix-vector multiplication, which is the dominant operation in an iterative solution method, can be reduced from $O((N+1)^2n^2)$ to $O((N+1)n^2)$. This allows medium-dimensional problems, such as the NHE model problem, to be considered and the use of the selection criterion in the VKOGA algorithm, without incurring extreme runtimes. As outlook, to work with high-dimensional problems, one could use MOR. In the numerical experiments, we mainly discussed the quality of control by the surrogate via the $MRL^2Error$. It should be noted that the error between the real VF and the surrogate could accumulate at each control step in the surrogate closed-loop control. However, this is not observed in the numerical experiments, where not only the stability of this closed-loop is seen, but also that the solutions are nearly optimal. Moreover, these surrogate-controlled solutions can be computed much faster than optimal open-loop solutions when the surrogate is available, as the runtimes show, leading to a significant offline-online decomposition. Furthermore, the numerical experiments show that the use of the context-aware structured Hermite surrogate is worth trying when a quadratic approximation to the VF is available. This is motivated by the result for the GST model problem.  In general, this data-based approach could be extended to OCPs with finite time horizons, such as those found in MPC. In this, the VF is dependent on the initial time, which raises the question of how the surrogate should be designed to deal with this additional time variable.
\section*{Acknowledgements}
The authors gratefully acknowledge the financial support of this project by the  International Research Training Group  2198 (IRTG) ''Soft Tissue Robotics''. Further, we thank the Deutsche Forschungsgemeinschaft (DFG, German Research Foundation) for supporting this work by funding - EXC2075 – 390740016 under Germany's Excellence Strategy.

\bibliographystyle{abbrv}
\bibliography{bibtex} 

\end{document}